\documentclass{amsart}
\usepackage{latexsym,amsmath,amsthm,amssymb}

\newtheorem{theorem}{Theorem}[section]
\newtheorem{lemma}[theorem]{Lemma}
\newtheorem{corollary}[theorem]{Corollary}
\newtheorem{proposition}[theorem]{Proposition}

\theoremstyle{remark}
\newtheorem{example}[theorem]{Example}
\newtheorem{remark}[theorem]{Remark}

\theoremstyle{definition}
\newtheorem{definition}[theorem]{Definition}

\newcommand{\dR}{\ensuremath{\mathbb{R}}} 
\newcommand{\R}{\dR}

\begin{document}

\title{Weak regularity of Gauss mass transport}

\author{Alexander V. Kolesnikov}

\begin{abstract}
Given two probability measures $\mu$ and $\nu$ we consider a mass
transportation mapping $T$ satisfying 1) $T$ sends $\mu$ to
$\nu$, 2) $T$ has the form $T = \varphi \frac{\nabla
\varphi}{|\nabla \varphi|}$, where $\varphi$ is a function with convex sublevel
sets.
 We prove a change of variables formula for $T$.
We also establish Sobolev estimates for $\varphi$, and  a new form of the parabolic
maximum principle.
 In addition, we discuss relations to the Monge--Kantorovich problem, curvature flows theory, and
parabolic nonlinear PDE's.
\end{abstract}

\maketitle
Keywords: optimal transportation, Monge--Kantorovich problem, Gauss curvature flows, parabolic
Monge--Amp{\`e}re equation, Alexandrov maximum principle, parabolic maximum principle,
Sobolev and H{\"o}lder a~priori estimates.

\section{Introduction}

 In this paper we study  a class of mass transportation
 mappings having the form
$$
T = \varphi \frac{\nabla \varphi}{|\nabla \varphi|}
$$
with some potential $\varphi$.
The mappings of this type have been
 introduced in \cite{BoKo}, \cite{BoKoR}.
 Assume we are given a couple of
probability measures $\mu = \rho_0 \ dx$
and $\nu = \rho_1 \ dx$.
It has been shown that, under general assumptions, there exists a unique
 $\varphi$ with convex sublevel sets $A_t = \{x \colon
\varphi(x) \le t\}$ such that
$$
T\colon x \to \varphi(x) \cdot {\rm n}(x),
$$
where ${\rm n}(x)$
is the normal vector to $\partial A_t$ at $x$ with $t=\varphi(x)$ and $T$
satisfies the equality $\nu = \mu \circ T^{-1}$. We point out that the restriction of $T$ to
every level set $\partial A_t$ coincides (up to the factor $t$) with
the Gauss map of $\partial A_t$. In what follows we use the name
``Gauss mass transport'' for $T$.

Mappings of this kind  are closely related to several areas of research.
They can be considered as ``parabolic'' analogs of  optimal transportation mappings,
which attract attention of researchers from
the most diverse fields, including probability, partial differential
equations, geometry, and infinite-dimensional analysis (see
 \cite{Vill}, \cite{Vill2}, and \cite{BoKo2005}).
 In addition, they arise  naturally in
 the Gauss curvature flow theory.
Concerning transformations of measures
 of other related types, see \cite{B}, \cite{B08}, \cite{BKM}.

The main goal of this paper is to establish some
 regularity properties of the mapping $T$.
More precisely, we prove that $T$  satisfies
a change of variables formula, which can be considered as
the weakest regularity property of $T$.

The corresponding result in the elliptic case (optimal mappings
between measures with densities always satisfy a change of variables formula)
belongs to McCann \cite{McCann2}.
This result turns out to be quite
useful for different applications.
Applications of the change of variables formula include, for instance,
the so-called above-tangent formalism  which is  a crucial technique
in variational problems, PDE's, and probability (see \cite{Vill2}, \cite{AGS}, \cite{BarKol}).

The  paper is organized as follows.

In Section 2, we briefly describe the main results of \cite{BoKo} that are used
throughout. These are the results on existence and
uniqueness of Gauss maps, a description of an important scaling procedure, and certain duality
relations. In addition, we describe the relations to
 curvature flows and the parabolic Monge--Amp{\`e}re equation.

Our main result is proved in  Section 3. We show that $T$ satisfies
 the following change of variables formula:
 $$
 \rho_1 \circ T \cdot \mathcal{J} = \rho_0 \  \ \mbox{with} \ \ \mathcal{J} = \det D_a T,
 $$
 where $D_a T$ can be understood as the absolutely continuous part of the distributional
 derivative of $T$. One has
 $$
\mathcal{J} = \varphi^{d-1} |D_a \varphi| K,
 $$
where $|D_a \varphi|$ is the absolutely continuous part of the full variation of the vector-valued measure $\nabla \varphi$
and $K$ is the Gauss curvature of $\partial A_{\varphi(x)}$.

In Section 4 we establish some natural Sobolev a-priori estimates for $\varphi$.
We emphasize that $\varphi$ is not Sobolev but only BV in general.
Under assumption that $\rho_1= \frac{C}{r^{d-1}}$  we show that for every $p>0$
$$
C_{p,R}
\int_{A} |\nabla \varphi|^{p+1} \ d \mu
\le
 \int_{A}  \Big| \frac{\nabla \rho_{0}}{\rho_{0}}  \Big|^{p+1} \ d \mu
+
\int_{\partial A} K^{-p} \rho^{p+1}_{0} \ d \mathcal{H}^{d-1}.
$$

 Another natural question arising in the study of the Gauss mass transport
 is the validity of some parabolic analogs  of the  maximum principle.
 Applying  the  mass transportation arguments
 one can establish (see Section~5)
 the following form of the parabolic maximum principle:
  every smooth function $f$ on a convex set $A$ satisfies the inequality
 $$
 \sup_{A} f \le \sup_{\partial A} f + \ C(d) \int_{\mathcal{C}_{-f,l}} |\nabla f| K \ dx,
 $$
 where
 $$ \mathcal{C}_{-f,l} = \{x\colon x \in A_t \cap \partial \ \mbox{conv}(A_t) \}, \ \
 A_t = \{-f \le t\}$$
 is  the set of contact points for the sublevel sets of $-f$, $ \mbox{conv}(A_t)$ is the
 convex envelope of $A_t$,
 and $K$ is the corresponding Gauss curvature.
 This estimate is naturally related to the Gauss mass transport
 and  the second-order
 nonlinear parabolic differential operator $ f  \mapsto  |\nabla f| K$
 (similarly to the Monge--Amp{\`e}re operator in the classical maximum principle).
 The inverse mapping $S=T^{-1}$
 is associated with another parabolic differential operator:
 $$
 f \mapsto \frac{f_r \cdot \det \bigl(f \cdot \mbox{Id} + D^2_{\theta} f\bigr)}{r^{d-1}},
 $$  where
 $D^2_{\theta} f$ is the Hessian on $S^{d-1}$.
The corresponding maximum principle is proved.

In Section 7, we are concerned with the regularity of the
parabolic Monge--Amp{\`e}re  equation. In particular, we briefly
explain how the arguments employed in \cite{Tso}
can be extended to our situation to prove H{\"o}lder's regularity
of $\varphi$. Thus we establish H{\"o}lder's continuity of $\varphi$ assuming
that $\rho_1, \rho_2 \in C^{2,\alpha}(A)$
and $\partial A$ is smooth and uniformly convex.

The author express his gratitude to Vladimir Bogachev
for valuable suggestions and remarks.

\section{Existence and basic properties}

In what follows we denote by $\mathcal{H}^{m}$ the $m$-dimensional
Hausdorff measure on $\R^d$, $m \le d$.
For Lebesgue measure we also use common notation $\lambda$. We denote by
$S^{d-1}$ the unit sphere in $\R^{d}$ (and by $S^{d-1}_{+}$ its upper-half).
We also use the symbols $D_{\theta}$, $D^2_{\theta}$ for the gradient and the Hessian on $S^{d-1}$.

It will be assumed throughout the paper that
\begin{itemize}
\item[A1)] the measure $\mu$ is supported on a compact convex  set $A$
\item[A2)] the measure $\nu$ is supported on $B_R = \{ x\colon |x| \le R\}$ for some $R>0$
\item[A3)] the measure $\mu$ is absolutely continuous
with respect to $\lambda|_A$ and $\nu$ is absolutely continuous
with respect to  $\lambda|_{B_{R}}$.
\end{itemize}

We start with a brief outline of two areas of research closely related to the
Gauss mass transport.

1) Optimal transportation.

 Optimal transportation can be described as a problem
of optimization of a certain functional associated
with a pair of measures.
The quadratic transportation cost   $W_2^2(\mu, \nu) $  between two
probability measures $\mu, \nu$ on $\mathbb{R}^d$ is
 defined as the minimum of the
Kantorovich functional:
\begin{equation}
\label{MKpr}
m \mapsto \int_{\mathbb{R}^d\times \mathbb{R}^d}
|x_1 - x_2|^2 \,d m(x_1,x_2), \quad m\in \mathcal{P}(\mu,\nu),
\end{equation}
where $\mathcal{P}(\mu,\nu)$ is the set of all probability measures
on $\mathbb{R}^d\times\mathbb{R}^d$ with the marginals $\mu$ and~$\nu$;
here $|v|$ denotes the Euclidean norm of $v\in \mathbb{R}^d$. The problem of
minimizing (\ref{MKpr}) is called the mass transportation problem.
In many cases there exists a mapping
$T\colon\, \mathbb{R}^d \to \mathbb{R}^d$, called
the optimal transport between $\mu$ and $\nu$,
such that $\nu = \mu \circ T^{-1}$ and
$$
W^2_2(\mu, \nu) =\int_{\mathbb{R}^d} |x -T(x)|^2 \, \mu(dx).
$$
If $\mu$ and $\nu$ are absolutely continuous, then, as shown by
Brenier   and McCann (see \cite{Vill}),
 there exists an optimal transportation $T$ which takes
 $\mu$ to~$\nu$. Moreover,
this mapping is $\mu$-unique  and has the form $T =
\nabla W$, where $W$ is convex. Assuming smoothness of   $W$, one can easily verify that
$W$  solves the
following nonlinear PDE (the  Monge--Amp{\`e}re  equation):
$$
\det D^2 W = \frac{\rho_0}{\rho_1(\nabla W)}.
$$
In fact, this equation is satisfied in a certain
sense without any smoothness assumptions
(see Section 3).

  2) Geometric flows.

  We refer to  \cite{Gerh}, \cite{Giga}   for an account in geometric flows.
  Let $\{\Gamma_t\} \subset \R^d$  be a  family of embedded
  hypersurfaces. Denote by $V(x,t)$ the velocity in the direction of
  the inward normal $-{\rm n}(x)$ at a point $x \in \Gamma_t$.
  We say that $\{\Gamma_t\}$
  satisfies a surface evolution equation (or $\{\Gamma_t\}$
  is a geometric flow) if $V$ satisfies
  \begin{equation}
  \label{vel-eq}
  V = f(x,t,\rm n, D\rm n)
  \end{equation}
 for some given function $f$.
If $f=H$ is the mean curvature, then ${\Gamma_t}$ is called  the
mean curvature flow. If $f=K$ is the Gauss curvature, then
${\Gamma_t}$ is called the Gauss curvature flow.

The Gauss curvature flows have been introduced by Firey
\cite{Fir} as a model of the wearing stone on a beach.
The existence and uniqueness of
a Gauss curvature flow starting from a smooth
initial convex surface has been obtained by Tso \cite{Tso} by solving a corresponding
parabolic Monge--Amp{\`e}re equation.
He proved, in particular, that
$\Gamma_t$ remains convex and shrinks to a point in finite time.
The same result for the mean curvature flow  has been obtained by
Huisken \cite{Huis}.
More
on Gauss curvature flows see  in \cite{BenAnd}.

The main problem arising in respect with non-convex
initial surfaces is the eventual singularity of the solution. It turns out that in
general  $\Gamma_t$ becomes singular in finite time. To overcome
this problem several notions of generalized solutions have been
proposed.
A  weak notion of a solution to (\ref{vel-eq}) have been
introduced by Brakke \cite{Brakk}. He proved the existence
of the mean curvature flow  for any initial data in some
generalized measure-theoretical sense. According to the  level-set
method (see \cite{Giga}),  the family $\{\Gamma_t\}$ is considered
as level sets of some function $u(t,x)$ satisfying a
nonlinear parabolic equation in viscosity sense. Finally, it is
known that sometimes the solutions to curvature flows can be
obtained as scaling limits of certain elliptic or parabolic
equations. For instance, the mean curvature flow can be obtained
as a singular limit of the solutions to Allen--Cahn or
Ginzburg--Landau equations
(see \cite{Ilmanen}, \cite{Soner}).
It has been shown in  \cite{BoKo}
that Gauss curvature flows starting from convex surfaces are
 singular limits of
some optimal transportation problems.
More precisely, the following result  has been proved in \cite{BoKo}.

{\bf Theorem.} { \it Let $A \subset \mathbb{R}^d$ be a compact
convex set   and let $\mu = \rho_0\, dx$ be a probability measure
on $A$ equivalent to the restriction of Lebesgue measure. Let $\nu
= \rho_1\, dx$ be a probability measure on $B_R = \{x\colon\, |x|
\le R\}$ equivalent to the restriction of Lebesgue measure. Then,
there exist a Borel mapping $T\colon\, A \to B_R$ and a continuous
function $\varphi\colon\, A \to [0,R]$ with convex sub-level sets
$A_s = \{\varphi \le s\}$
 such that $\nu = \mu \circ T^{-1}$
and
$$
T = \varphi \cdot {\rm n}
\quad
\hbox{$\mathcal{H}^d$-almost everywhere,}
$$
  where
${\rm n} = {\rm n}(x)$ is a
unit outer normal vector to the level set
$\{y\colon\, \varphi(y)=\varphi(x)\}$ at the point~$x$.

If $\varphi$ is smooth, the level sets of $\varphi$
are moving according to
the following Gauss curvature flow equation:
\begin{equation}
\label{GaussFlow} \dot{x}(s) = -s^{d-1} \ \frac{\rho_1(s {\rm
n})}{\rho_0(x)}  K(x)\cdot {\rm n}(x)
\end{equation}
where $x(s) \in \partial A_{R-s}$,  $0 \le s \le r$, $x(0) \in
\partial A$ is any initial point satisfying $\varphi(x(0))=R$. }

\begin{remark}
\begin{itemize}
\item[1)] The theorem does not guarantee that the boundary $\partial A$
is exactly the level set $\{\varphi =R\}$. Nevertheless, one can
easily check that this is indeed the case when $A$ is strictly
convex.

\item[2)] It is not clear in general whether $\{x\colon \varphi(x)=0\}$
contains a unique point or just has Lebesgue measure zero.

\item[3)] The case $\rho_{1} = \frac{1}{r^{d-1}}$, $\rho_{0}=C$ corresponds
to the standard Gauss curvature flow. The asymptotic behavior of
$\partial A_{r}$
 for small values of $r$ is a standard problem in differential
 geometry. For the classical Gauss flow it is known that $\partial A_{r}$
 is asymptotically spherical in shape for values of $r$ close to $0$ (see \cite{BenAnd}). This problem has not
 been studied so far for the
 flows of the type (\ref{GaussFlow}).

 \item[4)]
 Potential $\varphi$ is not  Sobolev in general, but admits a
 bounded variation (BV) (see \cite{AFP}). The distributional derivative of $\varphi$ can have a singular
 component in the ${\rm{n}}$-direction.
\end{itemize}
\end{remark}

In addition,  $T$ is unique and admits an inverse
$T^{-1}$ (see \cite{BoKo}, Section 3).

Let us briefly describe the idea of the proof and some important
related facts. The potential $\varphi$ is a pointwise limit of a
sequence of functions $\{\varphi_t\}$ with convex sublevel sets.
To construct $\varphi_t$ we consider the optimal transportation
$\nabla W_t$ of $\mu$ to $\nu \circ S^{-1}_t$, where $S_t(x) = x
|x|^t$. Let us set
$$
T_t = \frac{\nabla W_t}{|\nabla W_t|^{\frac{t}{1+t}}}
$$
Clearly,  $T_t$ pushes forward $\mu$ to $\nu$.
Choose $W_t$ in such a way that $\min_{x \in A} W_t(x) =0$.
Define  a new potential function $\varphi_t$ by
$$
W_t = \frac{1}{t+2} \varphi^{t+2}_t.
$$
Then one has
$$
T_t = \varphi_t \frac{\nabla \varphi_t}{|\nabla \varphi_t|^{\frac{t}{t+1}}}.
$$
It was shown in \cite{BoKo} that
$$
\lim_{n\to \infty}\varphi_{t_n} = \varphi, \
\lim_{n\to \infty}T_{t_n} = T
$$
 almost everywhere (for a suitable subsequence $t_n\to\infty$).

The dual potentials
$$W^{*}_t(y) = \sup_{x \in \R^d} \bigl( \langle x, y \rangle - W_{t}(x) \bigr)$$
of the corresponding dual Monge--Kantorowich problem define via
renormalization another natural convergent sequence
$$
H_t(y) = \frac{W^{*}_{t}(y|y|^t)}{|y|^{1+t}}.
$$
It was shown in \cite{BoKo} that
$$
H_t \to H
$$
pointwise, where
$$
H(r, \theta) \colon B_R = [0,R] \times S^{d-1} \to  \R,
$$
$$
H(r, \theta) =  \sup_{x\colon \varphi(x) \le r} \langle \theta,x \rangle.
$$
We warn the reader that in \cite{BoKo} we deal with a slightly different potential
$\psi_t = H_t(r, \theta) r$.

Let us describe the expression for $T^{-1}_t$ in terms of $H_t$.
To this end we fix ${\rm n} \in S^{d-1}$ and introduce local
coordinates $(\theta_1, \cdots, \theta_{d-1})$ on $S^{d-1}$ in a
neighborhood of ${\rm n}$. We assume everywhere below that
$$
e_i = \frac{\partial {\rm n} }{\partial \theta_i}
$$
constitute an orthonormal basis in the tangent space of $S^{d-1}$ at $\rm{n}$.
Then  the following relation holds
$$
T^{-1}_t(y) = \Bigl(H_t + \frac{r}{t+1} (H_t)_r \Bigr) \cdot {\rm n}
+ \sum_{i=1}^{d-1} (H_t)_{\theta_i} \cdot e_i.
$$
In the limit $t \to \infty$ one has
$$
T^{-1}(y) = H \cdot {\rm n}
+ \sum_{i=1}^{d-1} H_{\theta_i} \cdot e_i
= H \cdot {\rm n}
+ D_{\theta} H.
$$

\begin{remark}
In what follows we often choose the following convenient local
coordinate system on $S^{d-1}$. We take  the center of $S^{d-1}$
for the origin and introduce the standard Euqlidean coordinates in
$\R^d$ such that ${\rm n}$ becomes the North Pole: ${\rm n} =
(0,0, \ldots, 1)$. A neighborhood of ${\rm n}$ can be parametrized
by
$$
(\theta_1, \ldots, \theta_{d-1}) \to \Bigl(\theta_1, \ldots, \theta_{d-1},
(1-\sum_{i=1}^{d-1} \theta^2_i \bigr)^{1/2} \Bigr).
$$
In particular, one has  at ${\rm n}$:
$$
\frac{\partial e_i }{\partial \theta_i} = -{\rm n}, \ \frac{\partial e_j }{\partial \theta_i} = 0, \ i \ne j.
$$
Clearly, $(r, \theta_1, \ldots, \theta_{d-1})$ is a parametrization of a cone with the
vertex at the origin.
\end{remark}

Now we describe the relation between the
Gauss mass transport and the parabolic Monge--Amp{\`e}re  equation.

Several parabolic analogs  of the elliptic Monge--Amp{\`e}re  equation
have been introduced by Krylov (see \cite{Kryl}).
He also proved some forms of the parabolic
maximum principle (see also \cite{Tso2}).

Let $\mu =  \rho_0\, dx$ be a probability measure on an strictly
convex set $A$. Consider a Gauss mass transportation
$$
T=\varphi \frac{\nabla \varphi}{|\nabla \varphi|}
$$
sending $\mu$ to a measure $\nu = \rho_1\, dx$ on $B_R \colon=
\{x\colon\, |x| \le R\}$.

\begin{example}
Assume $d=2$ and fix a standard coordinate system $(x_1, x_2)$. Assume that  the functions  below  are smooth.
Introduce the global polar coordinate system $(r, \theta)$.
One has
$$
T^{-1} = H \cdot {\rm n} + H_{\theta}  \cdot {\rm v},
$$
$$
{\rm n} = \bigl(  \cos \theta,  \sin \theta \bigr), \ {\rm v} = (- \sin \theta,  \cos \theta).
$$
Let us compute the derivative of $T^{-1}$ in polar coordinates:
$$
T^{-1}_r = H_r \cdot {\rm n} + H_{\theta r}  \cdot {\rm v}
$$
$$
T^{-1}_{\theta} = H \cdot \dot{\rm n}_{\theta} +    H_{\theta} \cdot {\rm n} +
H_{\theta} \cdot \dot{\rm v}_{\theta} +  H_{\theta \theta}  \cdot {\rm v}
= (H + H_{\theta \theta}) \cdot {\rm v}.
$$
Taking into account that
 $
 \det D(r,\theta) = \frac{1}{r}
 $
one gets
$
\det D T^{-1} = \frac{H_r(H+H_{\theta \theta})}{r}.
$
Finally, by the change of variables formula
\begin{equation}
\label{MA0}
\rho_1 = \rho_0( H \cdot {\rm n} + H_{\theta}  \cdot {\rm v}) \frac{H_r(H+H_{\theta \theta})}{r}.
\end{equation}

Let us describe a standard trick
which allows to rewrite (\ref{MA0})
in the form of the parabolic Monge--Amp{\`e}re  equation.
Introduce another variable
on
$x_2 <0$:
$$
z = - \mbox{ctg} \ \theta, \ \pi \le \theta \le 2 \pi.
$$
Thus $\theta = \mbox{arcctg} (-z)$.
Instead of $H$ it is convenient to work with
$$
u = \sqrt{1+z^2} \ H.
$$
Note that $u$ is just the restriction of the corresponding $1$-homogeneous
support function $H_{A_r}$ with a fixed $r$ to the line $x_2 =-1$.
In particular, $u$ is convex in $z$.
Taking into account that
$
\frac{\partial}{\partial z} = \frac{1}{1+z^2} \frac{\partial}{\partial \theta},
$
one can easily compute
$$
u_{z} =   \frac{zH + H_{\theta}}{\sqrt{1+z^2}},
\ u_{zz} = \frac{H + H_{\theta \theta}}{(1+z^2)^{\frac{3}{2}}}.
$$
Finally,
we set
$$
\mathcal{T} = T^{-1} \circ (r, \mbox{arcctg}(-z)).
$$
Writing this mapping in  coordinates $(x_1,x_2)$ as a function of
$(z,r)$, one gets
$$
\mathcal{T}^{-1} = (u_z, zu_z-u) =(u_z, u^{*}(u_z)),
$$
where $u^{*}$ is convex conjugated to $u$ with respect to
$z$-variable
$$
u^{*}(z,r) = \sup_{x \in \R^{1}} \bigl( xz - u(z,r) \bigr).
$$
The change of variables formula takes the form
\begin{equation}
\label{MA1}
 u_r \cdot u_{zz}  = \frac{1}{ \rho_0(  u_z, zu_z-u)} \frac{r}{1+z^2} \rho_1 \Bigl( \frac{rz, -r}{\sqrt{1+z^2}}\Bigr),
 \ (z, r) \in \R^{+} \times \R.
\end{equation}

Note that (\ref{MA1}) can be considered as a {\it parabolic Mong{e}-Amp{\`e}re}  equation.

In addition, (\ref{MA1}) can be easily interpreted from the point of view
of mass transportation. Indeed, let us set
$$
\tilde{\nu} =  \frac{r}{1+z^2} \rho_1 \Bigl( \frac{rz,-r}{\sqrt{1+z^2}}\Bigr) \ drdz.
$$
Then $\tilde{\nu}$ is a measure on $\R \times [0,R]$ which coincides with
 the image of $\nu$ under the mapping
$$
(x,y) \longmapsto \frac{(rz,-r)}{\sqrt{1+z^2}} .
$$
Further, $\mu$ is the image of $\tilde{\nu}$ under
${\mathcal T}^{-1}$. Function $u$ is convex in $z$ and increasing
with respect to~$r$.
\end{example}

All these computations can be generalized
to the multidimensional case. One has
$
T^{-1} = H \cdot {{\rm n}} + \sum_{i=1}^{d-1} H_{\theta_i} \cdot e_i
= H \cdot {{\rm n}} + D_{\theta} H
$
and
\begin{equation}
\label{MA0+}
\rho_1 = \rho_0( T^{-1}) \frac{H_r \cdot \mbox{det}(H \cdot \mbox{Id} + D^2_{\theta} H)}{r^{d-1}}.
\end{equation}
Here $D^2_{\theta} H$
denotes the Hessian of $H$
on the unit sphere. For computing $D^2_{\theta} H$ it is convenient
to deal with
the local polar coordinate system as described above.
In this case $ D^2_{\theta} H({\rm n})$ can be represented
just by the matrix $(\partial^2_{\theta_i \theta_j} H)$.
Note that
$$
\mbox{det}(H \cdot \mbox{Id} + D^2_{\theta} H) = \frac{1}{K(T^{-1})}, ~ H_r = \frac{1}{|\nabla \varphi(T^{-1})|} .
$$
Finally, let us define  coordinates $(z,r)$ and the corresponding chart
$$V(z,r) \colon \{ -R < x_{d} < 0\} \to B_{R},$$
$$
\bigl(x_1,
\ldots, x_{d}  \bigr) = \frac{r}{\sqrt{1 + z_1^2+ \cdots +
z_{d-1}^2}} \bigl( z_1, \ldots, z_{d-1},-1 \bigr) = V(z,r).
$$
Now we introduce a new potential $u$
$$u = \sqrt{1 + z_1^2+ \cdots + z_{d-1}^2} \ H$$
and verify the following proposition by direct computations.

\begin{proposition}
\label{sphere-plain}  Assume that $T$ is smooth. The following
representations hold on $-R < x_d <0$:
\begin{itemize}
\item[1)]
$$
{\mathcal T}^{-1} = \bigl( \nabla_z u , \bigl<z, \nabla_z u \bigr> - u\bigr)
=  \bigl( \nabla_z u ,  u^*(\nabla_z u)\bigr),
$$
where
$$
u^{*}(z,r) = \sup_{x \in \R^{d-1}} \bigl( \langle x,z \rangle -
u(z,r) \bigr),
$$
$$
{\mathcal T}^{-1} = T^{-1} \circ V(z,r).
$$
\item[2)]
$$
\det \bigl( H \cdot \mbox{\rm Id}
+ D^2_{\theta} H  \bigr) = (1+z_1^2 + \cdots +z^2_{d-1})^{\frac{3}{2}(d-1)} \det D^2_z u
$$
\item[3)]
the change of variables takes the form
$$
u_r \cdot \det{D^2_z u} =\frac{\tilde{\rho}_1}{\rho_0(\nabla_z u, \bigl<z,
\nabla_z u \bigr> - u)},
$$
where
$$
\tilde{\rho}_1
 = \frac{r^{d-1}}{\big(1 + z^2_1 + \cdots + z^2_{d-1}\bigr)^{\frac{3}{2} d -2}}\rho_1 
\Bigl( \frac{rz_1, \cdots, r z_{d-1}, -r}{\sqrt{1+ z^2_1 + \cdots + z^2_{d-1}}} \Bigr) $$
\end{itemize}
\end{proposition}

More on the parabolic Monge--Amp{\`e}re  equation
see in Section 7.

\section{Change of variables}

Let $A$ be any convex compact set of positive volume
and let $T \colon A \to B$ be a Gauss mass transport between two given probability measures $\mu$
and $\nu$ satisfying the assumptions specified in the introduction.
To prove the change of variables formula for the
Gauss mass transport we need to define the Gauss curvature
for sufficiently "large" amount of points $x \in \partial A$.
To this end we consider the corresponding support function
$$H_{A}(\theta) = \sup_{x \in A} \langle \theta, x \rangle.$$
Here we assume that $\theta \in \R^d$. Clearly, $H_{A}$ is
$1$-homogeneous and convex. Hence, by the Alexandrov theorem $H_A$
is almost everywhere twice differentiable. Recall that  every
convex function $V$ is a.e. twice differentiable in the Alexandrov
sense, i.e. for almost all $x$ there exists a matrix $D^2_a V(x)$
(the absolutely continuous part of the second distributional
derivative) such that
\begin{equation}
\label{A2der}
\bigl|V(y)  - V(x) - \langle \nabla V(x), y-x\rangle
- \frac{1}{2} \langle D^2_a V(x) \ y-x, y-x \rangle \bigr| = o(|y-x|^2), \ y \to x
\end{equation}
(see \cite{EG}).

\begin{remark}
A parabolic analog of the Alexandrov theorem
for  monotone-convex functions was proved
by Krylov (see  \cite{Kryl}).
\end{remark}

\begin{definition}
In what follows we say that $f \colon M \to \R$, where  $M$ is a Borel set
is differentiable at $x \in M$ in the sense of Alexandrov if (\ref{A2der}) holds for $y \in M$.
\end{definition}

In particular, homogeneity implies that  for every fixed $r>0$ the function
$H_A|_{\partial B_r}$ is twice differentiable
for $\mathcal{H}^{d-1}$-almost all $x \in \partial B_r$.

Recall that $H$ is defined as follows:
$$
H(r,\theta) = \sup_{\theta \in S^d, x \in A_r}
\langle \theta, x \rangle.
$$

\begin{lemma}
\label{H}
For $\mu$-almost all $x \in A$ and all $0 < r \le R$ the
function $H|_{\partial B_r}$ is twice differentiable at $r \cdot \rm{n}(x)$ in the Alexandrov sense.
\end{lemma}
\begin{proof}
It was noted above that $H|_{\partial B_r}$
is twice differentiable in the Alexandrov sense
for $\mathcal{H}^{d-1}$-almost all $y \in \partial B_r$.
Hence by  Fubuni's theorem the set of all $y$ such that
$H|_{\partial B_r}(y)$, $r=|y|$ is not twice differentiable
in the Alexandrov sense has $\nu$-measure zero. The claim follows from the fact that
$T$ pushes forward $\mu$ to $\nu$.
\end{proof}

Next we want to define the Gauss curvature for an arbitrary convex surface $\partial A$
$\mathcal{H}^{d-1}$-almost everywhere.
Let us recall how the Gauss curvature can be defined in the smooth case.

Assume that $\partial A$ is a level set of some smooth function $F$. Then
$$
{\rm n}(x) \colon=\frac{\nabla F(x) }{ |\nabla F(x)|}.
$$
Let $\{ e_1, e_2, \ldots, e_{d-1}\}$ be an orthonormal basis such that ${\rm n} \bot e_i$
for every $1 \le i \le d-1$.
Then
$$
K(x) = \det D {\rm n} (x),
$$
where $D$ is  the differential
operator  on the tangent space to $\partial A$.
Computing this expression in Euclidean coordinates one gets
$$
K = \det A_{i,j}, \ \mbox {where }
$$
\begin{equation}
\label{compK}
A_{i,j}(x)=
\frac{1}{|\nabla F(x)|}  \langle D^2 F(x) e_i, e_j \rangle, \ 1 \le i,j \le d-1.
\end{equation}
In particular, this formula is applicable
when the surface
is represented locally as the graph of a convex  function
$$
F = W(x_1, \ldots, x_{d-1}) - x_d,
$$
where $\{e_i\}$ can be obtained by an orthogonalization procedure from
a basis tangent to $F$ at some point.

It is convenient to compute the Gauss curvature in terms of the support function.
The following lemma (well known for smooth surfaces) gives another
practical way of computing.

\begin{lemma}
\label{W-H}
 Let $\partial A$ be a convex surface which coincides with a graph
of some convex function $W(x_1, \ldots, x_{d-1})$ in a neighborhood
$\Omega$ of $x_0$ and $\rm{n}(x_0)$ is unique at~$x_0$.
Then the following are equivalent
\begin{itemize}
\item[1)]
$W$ is differentiable at $x_0$ in the Alexandrov sense and $D^2_a W (x_0)$
is nondegenerate,

\item[2)]
$H|_{S_{d-1}}$ is differentiable at ${\rm n}(x_0)$ in the Alexandrov sense
and
$$
{\det} \Bigl(H \cdot \mbox{\rm Id} +   \bigl(D^2_{\theta}\bigr)_a H \Bigr) \circ {\rm n}(x_0) \ne 0.
$$
\end{itemize}
\end{lemma}
\begin{proof}
Choosing an appropriate coordinate system, we may assume without loss of generality that
$x_0=0$,
$W(0)=0$ and $\nabla W(0)=0$.
Let us assume for a while that $W$ is smooth and $D^2 W>0$ in $\Omega$. Introduce local coordinates $(\theta_1, \cdots, \theta_{d-1})$
on $S^{d-1}$ satisfying the equality
$$
T_W = \frac{1}{\sqrt{1+|\nabla W|^2}} \Bigl(W_{x_1}, \ldots, W_{x_{d-1}}, -1 \Bigr)
=
\Bigl(\theta_1, \theta_2, \ldots, \theta_{d-1}, - \sqrt{1-\sum_{i=1}^{d-1} \theta^2_{i}} \Bigr).
$$
Clearly, the first $d-1$ basis vectors of the ambient space constitute an orthogonal basis
in the tangent space to $S^{d-1}$ at $x_0$.
Note that
$$
\nabla W(x) = \frac{\theta}{\sqrt{1-|\theta|^2}},
$$
where $\theta=(\theta_1, \ldots, \theta_{d-1})$, $x=(x_1, \ldots, x_d)$
and
$$
H = \sum_{i=1}^{d-1} \theta_i x_i - \sqrt{1-|\theta|^2} x_d.
$$
Set $W^{*}(x) = \sup_{y \in \Omega} (\langle x,y\rangle -W(y))$.
Since $\nabla W$ and $\nabla W^{*}$ are reciprocal, one has
$$
x= \nabla W^{*} \Bigl( \frac{\theta}{\sqrt{1-|\theta|^2}}\Bigr).
$$
Taking into account that $x_d = W(x)$, one has
$$
H =   \Bigl\langle \theta,  \nabla W^{*} \Bigl( \frac{\theta}{\sqrt{1-|\theta|^2}}\Bigr) \Bigl\rangle
-
 \sqrt{1-|\theta|^2} \cdot W \Bigl(  \nabla W^{*} \Bigl( \frac{\theta}{\sqrt{1-|\theta|^2}}\Bigr) \Bigr).
$$
Hence
$$
H(\theta)= \sqrt{1-|\theta|^2} \ W^{*} \Bigl(  \frac{\theta}{\sqrt{1-|\theta|^2}} \Bigr)
$$
on $T_W(\Omega)$.
This is equivalent to
$$
W^{*}(x) = \sqrt{1+|x|^2} \ H \Bigl(\frac{(x,-1)}{\sqrt{1+|x|^2}} \Bigr)
$$
on $\Omega$.
By approximation arguments these  relations remain valid for every convex $W$
in $\Omega$.
Now assume that $H$ is twice Alexandrov differentiable at $0$. Clearly, $H(0)=0$, $\nabla H(0)=0$.
The same holds for $W^{*}$. Using Alexandrov differentiability of $H$, we get
 $$W^{*}(x) =  \Bigl( \frac{x^2}{2} H(0,-1) + \frac{1}{2}
 \langle \bigl(D^2_{\theta}\bigr)_a H(0,-1)  x, x \rangle \Bigr) + o(x^2).$$
This means that $D^2_a W^{*} = H \cdot \mbox{Id} + \bigl(D^2_{\theta}\bigr)_a H$.
 We get 1) by the duality relations for convex functions (see, for instance, \cite{McCann2}). The opposite
implication follows by the same arguments.
\end{proof}

Clearly, if the surface is smooth and strictly convex, in the
situation of the Lemma \ref{W-H} one has
$$
K = \frac{1}{\det \bigl( H \cdot \mbox{Id} + D^2_{\theta} H \bigr) \circ {\rm n} }.
$$

\begin{definition}
Let $A$ be an arbitrary convex surface.
We call the following quantity $K$ "Gauss curvature  of $\partial A$
at $x$"
\begin{equation}
\label{KH}
K(x) \colon=  \frac{1}{\det \bigl( H \cdot \mbox{Id} + \bigl(D^2_{\theta}\bigr)_a H \bigr) \circ {\rm n(x)} }
\end{equation}
if there exists a unique normal ${\rm n(x)}$, $H$ is twice Alexandrov differentiable at ${\rm n(x)}$ and $H \cdot \mbox{Id} + \bigl(D^2_{\theta}\bigr)_a H$
is nondegenerate.

If ${\rm n(x)}$ is unique,  but $H$ is not twice Alexandrov differentiable
at ${\rm n(x)}$,
we set $K(x)=0$. The latter is equivalent to
$\det D^2_a W(x_1, \ldots, x_{d-1}) =0$ if $\partial A$ coincides locally with a graph of $W \colon \R^{d-1} \to \R$.
\end{definition}

\begin{remark}
Clearly, by Lemma \ref{W-H}
$K(x)$ is well-defined for $\mathcal{H}^{d-1}|_{\partial A}$-almost all $x$, since
$W$ is $\mathcal{H}^{d-1}$-a.e. differentiable.
\end{remark}

\begin{remark}
\label{compK-W}
In the special coordinate system considered in the proof
of Lemma \ref{W-H} one has $K = \det D^2_a W$.
Following the proof of Lemma \ref{W-H} one can  easily understand that (\ref{compK})
holds almost everywhere  in a non-smooth setting with $F=W-x_d$ if the second derivative of $W$ is understood in the Alexandrov
sense.
\end{remark}

Recall an important result of McCann \cite{McCann2}.

\vskip .1in

{\bf Theorem (McCann).}
{\bf (Change of variables formula for convex functions.)}
{\it\
Let $\mu = f \ dx$ and $\nu = g  \ dx$ be two probability measures
and $V$  be a convex function such that  $\nu=\mu \circ \nabla V^{-1}$.
Then for $\mu$-almost all $x$ one has
$$
g(\nabla V) \det D^2_a V = f,
$$
where $D^2_a V$ is the second Alexandrov derivative of $V$.
}

\vskip .1in

In the following proposition we deal with the Gauss map
${\rm n}\colon \partial A \to S^{d-1}$  (non multivaled!)  which is $\mathcal{H}^{d-1}$-a.e. well defined.

\begin{proposition}
\label{Sard}
For every  $A_t = \{x\colon \varphi(x) \le t\}$
the measure
 $\Bigl(K \cdot \mathcal {H}^{d-1}|_{\partial A_t} \Bigr)\circ \rm{n}^{-1}$ is absolutely continuous with respect to $\mathcal{H}^{d-1}$
and the following change of variables formula holds for every bounded Borel
function $f \colon S^{d-1} \to \R$:
$$
\int_{\partial A_t} f({\rm n}) K \ d\mathcal{H}^{d-1} = \int_{ { \rm n}(\partial A_t)} f  \ d\mathcal{H}^{d-1}.
$$
\end{proposition}
\begin{proof}

It is sufficient to prove this result for ${\partial A_t} \cap V$
instead of ${\partial A_t}$, where $V$ is a small neighborhood of
a point $x_0 \in {\partial A_t}$ with unique ${\mbox
\rm{n}}(x_0)$. Fix such a point and choose a coordinate system in
such a way that ${\mbox \rm{n}}(x_0) =  e_d$ and the surface
${\partial A_t}$ coincides (locally) with the graph of a convex
function $W \colon {\rm U} \subset \R^{d-1} \to \R$, where ${\rm U}$ is
an open ball containing $0$ and $W$ attains its minimum at $0$. In
addition, we may assume that $\partial W({\rm U})$ is a bounded
set. Let ${\rm U}$  be a local chart of ${\partial A_t}
\cap V$ and parametrize a part of ${\partial A_t} \cap V$ in the
following way
$$
{\rm U} \ni (x_1, \ldots , x_{d-1} ) \to (x_1, \ldots , x_{d-1}, W(x)).
$$
Since $W$ is Lipschitz on ${\rm U}$, the surface measure $\mathcal{H}^{d-1}$ on ${\partial A_t} \cap V$
can be computed in this chart by
$
m_0 = (1+ |\nabla W|^2)^{\frac{1}{2}} \mathcal{H}^{d-1}.
$
The Gauss map $\rm{n}$ is  given by
$$
\mbox{\rm{n}} = \frac{1}{\sqrt{1+|\nabla W|^2}} (-\partial_{x_1} W, \ldots , -\partial_{x_{d-1}} W, 1 ).
$$
This holds for almost every  $(x_1, \ldots , x_{d-1})$.

Identify the half-sphere $S^{d-1} \cap \{x_d \le 0\}$ with its projection $B^{d-1}_1$ on $(x_1, \ldots  x_{d-1})$
and $\rm{n}$ with the mapping $\tilde{\rm{n}} \colon  -\frac{\nabla W}{\sqrt{1+ |\nabla W|^2}}$ taking values in $B^{d-1}_1$. Note that
the surface measure on $S^{d-1}$ can be computed
in the local chart
$$(x_1, \ldots , x_{d-1}) \to (x_1, \ldots , x_{d-1}, \sqrt{1-x_1^2 - \ldots   -x^2_{d-1}})
$$
by
$
m_1 = \frac{1}{\sqrt{1-|x|^2}}  \mathcal{H}^{d-1}.
$

Note that $\tilde{\rm{n}} = F \circ \nabla W$ is the composition of $\nabla W$ with
the smooth mapping
$$F(x) =-\frac{x}{\sqrt{1+|x|^2}}$$
which is nondegenerate everywhere.

Writing the local chart expressions we get that the claim  is equivalent to
the equality $m_1|_{\tilde{\rm{n}}(M_{+})}
= (K \cdot m_0)|_{M_+} \circ \tilde{\rm{n}}$, where $M_{+} = \{x \in V \colon  \det D^2_a W >0 \}$.

Note that $K$ is well-defined on $M_{+}$. By Remark \ref{compK-W} we have
$$
K = \det D_a \tilde n = \det D_a W \cdot \det D F(x) \circ (\nabla W).
$$
By the result of McCann  the optimal transport $\nabla W$
pushes forward
$$\det D^2_a W \cdot \mathcal{H}^{d-1} |_{M_{+}}$$  to $\mathcal{H}^{d-1}|_{\nabla W(M_{+})}$.
Hence we obtain that
the image of the measure
$$K (1+ |\nabla W|^2)^{\frac{1}{2}} \mathcal{H}^{d-1}|_{M_{+}}
=
K \cdot m_0|_{M_{+}}
$$
under $\nabla W$ coincides with
$$
\det F(x) (1+|x|^2)^{\frac{1}{2}}  \mathcal{H}^{d-1}|_{\nabla W(M_{+})}.$$
Now applying the standard change of variables formula we get that the image of
$$
\det F(x) (1+|x|^2)^{\frac{1}{2}}  \mathcal{H}^{d-1}|_{\partial W(M_{+})}
$$
under $y = F(x)$ coincides with
$
\frac{1}{\sqrt{1-|y|^2}}  \mathcal{H}^{d-1}|_{\tilde{n}(M_+)} = m_1|_{\tilde{n}(M_+)}.
$
The proof is complete.
\end{proof}

The fact below follows easily from the McCann's theorem.

\begin{corollary}
\label{lusin}
If $V$ is a convex function satisfying $\det D^2_a V|_M=0$ for some set $M$ with $\lambda(M)>0$,
then the image of $\lambda|_M$ under $\nabla V$ is singular to
$\lambda$.
\end{corollary}

The proof of the following  lemma can be found,  for instance, in \cite{BoKo2005}.

\begin{lemma}
\label{18.03.09}
If $V$ is a convex function satisfying $D^2_a V>0$ on $M$, then
$
\lambda|_M \circ \nabla V^{-1}
$ is an absolutely continuous measure.
\end{lemma}

We prove an analog of the McCann's theorem for the
Gauss mass transport. We start with a  change of variables formula
for the mapping $\mathcal{T}^{-1}$ defined by a monotone-convex potential $u$.
Since $u$ and $H$ are related by a smooth change of variables, it gives immediately  a
change of variables formula for $H$.

\begin{remark}
We recall that
$u(z,r)$ (see Section 2) is convex in $z$ and increasing in $r$.
Hence one can define $u_r$ and $(D^2_z)_a u$ $\mathcal{H}^d$-almost everywhere,
where $u_r$ means a partial derivative of $u$ in the classical sense.
\end{remark}

\begin{theorem}
{\bf (Change of variables formula for $u$)} The potential $u$ satisfies the change of variables formula for
$\mathcal{H}^d$-almost all $(z,r) \subset \R^{d-1} \times [0,R]$
$$
u_r \cdot \det{(D^2_z)_a u}
=\frac{\tilde{\rho}_1}{\rho_0( \mathcal{T}^{-1})},
$$
where
$$
\tilde{\rho}_1
 = \frac{r^{d-1}}{\big(1 + z^2_1 + \cdots + z^2_{d-1}\bigr)^{\frac{3}{2} d -2}}\rho_1 
\Bigl( \frac{rz_1, \cdots, r z_{d-1}, -r}{\sqrt{1+ z^2_1 + \cdots + z^2_{d-1}}} \Bigr) $$
\end{theorem}
\begin{proof}
Fix an orthogonal coordinate system $(x_1, \ldots , x_d)$
and denote by ${\tilde{e}}_i$ the corresponding basis.
Recall that mapping ${\mathcal T}^{-1}$
sends $\tilde{\nu} = \tilde{\rho}_1 dx|_{\R^d \times [0,R]}$ to $\mu|_{\mathcal{T}^{-1}(\{x_d <0\})}$
and admits a.e. the representation
$$
\mathcal{T}^{-1} =
\sum_{i=1}^{d-1} u_{z_i} \cdot {\tilde e}_i + u^{*}(\nabla_z u) \cdot {\tilde e}_d,
$$
where $u^*(z) = \sup_{x \in \R^{d-1}} \bigl( \langle x, z \rangle - u(x) \bigr)$.
Let us represent ${\mathcal T}^{-1}$
as the composition of two mappings ${\mathcal T}^{-1} = S_2 \circ S_1$, where
$S_1 \colon  \R^{d-1} \times [0,R] \to \R^{d-1} \times [0,R]$ has the form
$$S_1(z,r)=(\nabla_z u,  r)$$ (all expressions are written in the
Euclidean $(z,r)$-coordinates!)
and
$$
S_2(z,r) = \sum_{i=1}^{d-1} z_i \cdot {\tilde e}_i + u^{*}(z,r) \cdot {\tilde e}_d.
$$
Let us show that $\det \bigl(D^2_z\bigr)_a u >0$ almost everywhere. Indeed, set
$$M \colon  = \{ (z,r) \colon  \det \bigl(D^2_z\bigr)_a u =0 \}.$$
Assume  that $\lambda(M)>0$. Then by Corollary \ref{lusin} and Fubini's theorem
$$\tilde{\nu} = \tilde{\rho}_1 \cdot \mathcal{H}^{d-1}|_M \circ S^{-1}_1$$ is a singular measure.
Let us disintegrate $\tilde{\nu}$ along the $r$-axis:
$$
\tilde{\nu}(r,z) = \nu^{z}(dr) \cdot \mu_0(dz).
$$
Here $\mu_0$ is the projection of $\tilde{\nu}$ onto $\R^{d-1}$ and  $\nu^{z}(dr)$
are the corresponding conditional measures.

Denote by $\tilde{\nu}_0$
the projection of $\tilde{\nu}$ onto $(z_1, \ldots , z_{d-1})$.
It follows from the relation $\tilde{\nu} \circ S^{-1}_2 = \mu|_{\mathcal{T}^{-1}(\{x_d <0\})}$
that the image of $\tilde{\nu}_0 = \bigl( \int \nu^{z}(dr)\bigr) \cdot \mu_0(dz)$
under
\begin{equation}
\label{10.07.09}
(z_1, \ldots , z_{d-1}) \to \sum_{i=1}^{d-1} z_i \cdot {\tilde e}_i
\end{equation}
coincides with the projection of $\mu|_{\mathcal{T}^{-1}(\{x_d <0\})}$ onto $(x_1, \ldots , x_{d-1})$.
 Since the latter admits a Lebesgue density and (\ref{10.07.09})
 is smooth and nondegenerate, one gets that $\mu_0(dz)$  admits a Lebesgue density
 $f_0(z)$.
Hence for $\mathcal{H}^{d-1}$-almost all $z$ the one-dimensional
measure $\nu^{z}(dr)|_{S_1(M)}$ is singular.
Note for $\mathcal{H}^{d-1}$-almost every fixed $z$ the mapping
$
r \mapsto u^{*}(z,r)
$
pushes forward $\nu^{z}(dr)|_{S_1(M)}$ to a one-dimensional absolutely continuous  measure.
Since $\nu^{z}(dr)|_{S_1(M)}$ is a singular measure, one has
$$u^{*}_r(r,z) =-\infty$$
(note that $u^{*}$ is decreasing in $r$) for $\mathcal{H}^{d-1}$-almost all $z$
and $\nu^{z}(dr)_{S_1(M)}$-almost all $r$. This follows by duality  from Corollary \ref{lusin}.

Next we note that
$$
u^{*}_r(\nabla_z u,r) = -u_r(z,r)
$$
$\lambda$-a.e.
Indeed, in the case of smooth  functions with non-degenerated second derivative it follows by differentiating
the duality relation
$$
u^{*}(\nabla_z u,r) + u = \langle \nabla_z u, z\rangle
$$
in $r$. The general case easily follows by approximations.

Thus $u_r = \infty$ ${\tilde \nu}$-almost surely on $M$. But this contradicts
 the assumption $\lambda(M)>0$.

Since  $\det \bigl(D^2_z\bigr)_a u >0$, by Lemma \ref{18.03.09}
the image $\tilde{\nu}$ under $S_1$ is absolutely continuous.
Then one can apply the   McCann's change of variables formula for $\mathcal{H}^1$-almost every fixed value of $r$.
Applying the same theorem once again  to $S_2$ (for $\mathcal{H}^{d-1}$-almost every fixed $z$) one  gets the
 result.
\end{proof}

\begin{corollary}
{\bf (Change of variables formula for $H$)}
Since ${\mathcal T}^{-1}$ and $T^{-1}$ are related by a smooth
change of variables, one immediately gets
$$
\rho_1 = \rho_0( T^{-1}) \frac{H_r \cdot \mbox{{\rm det}}(H \cdot \mbox{{\rm Id}} + \bigl(D^2_{\theta}\bigr)_a H)}{r^{d-1}}
$$
$\mathcal{H}^d$-almost everywhere on $B_R$.
\end{corollary}

\begin{corollary}
The Gauss curvature $K(x) = \mbox{{\rm det}}(H \cdot \mbox{{\rm Id}} + \bigl(D^2_{\theta}\bigr)_a H)$
 is well-defined and positive
for $\mu$-almost all $x$.
\end{corollary}

Recall that  $D\varphi$ denotes the generalized derivative of
$\varphi$ in the distributional sense.
Since $\varphi$ has convex
sublevel sets, it is a BV function (see \cite{AFP}). Hence $D
\varphi$ can be understood as a vector-valued measure satisfying
$$
\int \langle D \varphi, \xi \rangle \ dx = - \int  \varphi \ \mbox{div} \xi \ dx
$$
for every smooth compactly supported vector field $\xi$.
We denote by $\|D \varphi\|$ the corresponding total variational measure
and by $|D_a \varphi|$
its absolutely continuous component and by
 $|D_s \varphi|$ its singular component.

\begin{theorem}
{ \bf (Change of variables formula for $\varphi$)}
The following change of variables formula holds for $\mu$-almost all  $x \in A$:
$$
\bigl( K |D_a \varphi| \varphi^{d-1}\bigr) (x) \ \rho_1(T(x)) = \rho_0(x) .
$$
\end{theorem}
\begin{proof}
Let $\tilde{A} \subset A$ be a set, where $K$
is well-defined and positive. By the previous corollary $\mu(\tilde A)=1$.
Let us show that $|D_s \varphi|(\tilde{A})=0$. Indeed, otherwise we can
find a set $M_s \subset \tilde{A}$ with $\lambda(M_s)=0$
and $|D_s\varphi| (M_s)>0$.
By the coarea formula for BV functions (see \cite{AFP}, p.~159)
$$
0< \int_{{M_s}}  K \ d | D_s\varphi |
=
\int_{{M_s}}  K \ d  \|D \varphi  \|
=
\int_0^r \int_{\partial A_t \cap M_s}  K \ d \mathcal{H}^{d-1} dt.
$$
By Proposition \ref{Sard} and Fubini's theorem
the latter equals
$$
\int_0^r \int_{\partial B_t \cap T(M_s)} t^{-(d-1)} \ d \mathcal{H}^{d-1} dt
=
\int_{T(M_s)} |y|^{-(d-1)} d \mathcal{H}^{d}.
$$
Since $T$ pushes forward $\mu$ to $\nu$, one has  $\lambda(T(M_s))=0$.
We get a contradiction.

Applying again the coarea formula for BV functions we get
$$
\int_{{\tilde A}} \xi(T) K \ |D_a\varphi| dx =
\int_0^r \int_{\partial A_t \cap \tilde{A}} \xi(T) K \ d \mathcal{H}^{d-1} dt,
$$
for
 any Borel bounded function $\xi$.
By Proposition \ref{Sard}
$$
\int_{\partial A_t \cap \tilde{A}} \xi(T) K \ d \mathcal{H}^{d-1}
=
 \int_{B_t} \xi(y)  \ d \mathcal{H}^{d-1}(y)
$$
for almost all $t \in [0,R]$.
Since $T$ takes $\rho_0 dx$ to $\rho_1 dx$, one gets
$$
\int_{B_r} \xi(y)  \ \frac{K |D_a \varphi|}{\rho_0} \circ T^{-1}(y) \ \rho_1(y) dy
=
\int_0^r \Bigl( \int_{B_t} \xi(y)  \ d \mathcal{H}^{d-1}(y) \Bigr) |t|^{-(d-1)}dt,
$$
Hence, for $\mu$-almost all  $y \in B_r$, one has
$
\frac{K |D_a \varphi|}{\rho_0} \circ T^{-1}(y) \ \rho_1(y)
=
|y|^{-(d-1)}.
$
The proof is complete.
\end{proof}

\begin{corollary}
Comparing different change of variables formulae, one gets
$$
|D_a \varphi| = \frac{1}{H_r(T)}
$$
$\mu$-almost everywhere.
\end{corollary}

\section{Sobolev estimates for $\varphi$}

The main goal of this section is to establish some natural Sobolev estimates for $\varphi$
(Theorem \ref{sobolev-phi}). The proof is based on the integration-by-parts and change of variables
formulae.

Before proving Theorem \ref{sobolev-phi} we establish some
 $|\nabla \varphi|_{L^{\infty}}$-bounds
 with the help of the classical maximum principle. These estimates have an interest in their own,
 they will also serve as an intermediate step in Theorem \ref{Tso1}.

It will be assumed below that $H_r, H_{\theta}, H_{\theta \theta}$ are  continuous and continuously
differentiable in $r$ (except, maybe, the origin) up to the boundary.
We also assume without loss of generality that $H \ge 0$ and $H(0)=0$ (this can
be achieved just by shifting $A$
and assuming that $\varphi(0)=0$).
The estimates obtained below do not depend, however,
on higher derivatives (see in this respect Remark \ref{ap-est-ht}).

Let us set
$$
P = \rho_1 r^{d-1}.
$$
Since $H$ is smooth, it satisfies
$$
P = \rho_0( T^{-1}) \ H_r \cdot \mbox{det}(H \cdot \mbox{Id} + D^2_{\theta} H)
$$
up to $\partial B_R$.
We recall that $
H_r = 1/ |\nabla \varphi(T^{-1})|.
$

\begin{proposition}
\label{upper-bound}
a) Assume that for some $C>0$
$$
|\nabla \rho_0| \le C\rho^{1+\frac{1}{d}}_0, \ \  P \le C,
$$
and there exists $u \colon  (0,R] \rightarrow \R$ with $u \in L^1([a,R])$
for every $R>a>0$ such that
$$
\frac{P_{r}}{ P} \le u(r).
$$
In addition, assume that $\partial A$ is smooth,
$\lambda_0 =  \inf_{x \in \partial A} K(x) >0$
and $\rho_0|_{\partial A} \le C$, $P|_{\partial B_R} \ge \frac{1}{C}$.
Then
$$
H_r \ge D_1  \exp\bigl(-\int_r^R u(s)  ds\bigr).
$$
In particular
$$
|\nabla \varphi| \le D_2 \exp\bigl(\int_{\varphi}^R u(s) ds \bigr)
$$
with $D_1, D_2$ depending on $d, C, \lambda_0, R$.

b) Assume that for some $C>0$
$$
\frac{|\nabla \rho_0|}{\rho_0} \le C,  \ \rho_0 \ge \frac{1}{C}, \ P \le C
$$
and
$$
\frac{P_{r}}{ P} \ge  - C.
$$
In addition, assume that $\partial A$ is smooth, $\Lambda_0 =  \sup_{x \in \partial A} K(x) < \infty $,
 $\rho_0|_{\partial A} \ge C$, $P|_{\partial B_R} \le \frac{1}{C}$
and
$$
H \ge \varepsilon r
$$
for some $\varepsilon >0$.
Then
$$
H_r \le \frac{D_1}{r^{d}}, \ |\nabla \varphi| \ge D_2 \varphi^{d}
$$
with $D_1, D_2$ depending on $d, C, R, \varepsilon, \Lambda_0$.
\end{proposition}
\begin{proof}
a) We are looking for the minimum
of
$
H_r e^{f}
$
on $B_R \setminus B_{r_0}$ $r_0>0$, where $f=f(r)$ is a radially symmetric function to be chosen later.
Assume that the minimum is attained at some point $x_0  \notin \partial B_{R}$ .
We deal with the local coordinate system $(r,\theta)$  as described at Section 2.
Let us differentiate $\log H_r + f(r)$
along $r$ and every $\theta_i$ at this point.
One has
$$
H_{r\theta_i}=0, \ \frac{H_{rr}}{H_r} \ge -f_r.
$$
The second differentiation yields
$
H_{r \theta_i \theta_i} \ge 0.
$
Rotating the coordinate system when necessary we may assume that $D^2_{\theta \theta} H$
is diagonal at $x_0$.
Differentiating the change of variables formula in $r$ yields
$$
\frac{P_r}{P}
=
\frac{\langle \nabla \rho_0(T^{-1}), H_r \cdot \mbox{\rm n} + \sum_{i=1}^{d-1} H_{r\theta_i} \cdot e_i
\rangle}{\rho_0(T^{-1})}
+
\frac{H_{rr}}{H_r}
+ \sum_{i=1}^{d-1} \frac{H_r + H_{r \theta_i \theta_i}}{H + H_{\theta_i \theta_i}}.
$$
Hence
\begin{align*}
 \frac{P_r}{P} &
\ge
H_r \Bigl[ \frac{ \langle \nabla \rho_0(T^{-1}), n  \rangle}{ \rho_0(T^{-1})}  \Bigr]
-f_r +
H_r \sum_{i=1}^{d-1} \frac{1}{H + H_{\theta_i \theta_i}}
\\&
\ge
H_r \Bigl[ \frac{ \langle \nabla \rho_0(T^{-1}), n  \rangle}{ \rho_0(T^{-1})}  \Bigr]
-f_r +
(d-1) H_r  \Bigl[ \frac{1}{\det(H + D^2_{\theta \theta} H)} \Bigr]^{\frac{1}{d-1}}
\\&
\ge
 H_r \Bigl[ \frac{ \langle \nabla \rho_0(T^{-1}), n  \rangle}{ \rho_0(T^{-1})}  \Bigr]
-f_r +  (d-1)  H_r^{\frac{d}{d-1}} \Bigl[ \frac{ \rho_0( T^{-1})}{P} \Bigr]^{\frac{1}{d-1}}
.
\end{align*}

This implies
\begin{align*}
H^{\frac{d}{d-1}}_r
&
\le
\frac{1}{d-1} \Bigl[ \frac{P}{ \rho_0( T^{-1})} \Bigr]^{\frac{1}{d-1}}
 \Bigl[ \frac{P_r}{P} + f_r\Bigr]
\\&
- H_r \frac{ P^{\frac{1}{d-1}} } {(d-1)}
\Bigl[  \langle \nabla \rho_0(T^{-1}), n  \rangle \rho^{-\frac{d}{d-1}}_0 (T^{-1})\Bigr] .
\end{align*}

Applying  H{\"o}lder's inequality
one gets
\begin{align*}
H^{\frac{d}{d-1}}_r
&
\le \frac{1}{d-1}
\Bigl[ \frac{P}{ \rho_0( T^{-1})} \Bigr]^{\frac{1}{d-1}}
\Bigl[
 \frac{P_r}{P} + f_r \Bigr]
 +\frac{1}{2} H^{\frac{d}{d-1}}_r
 \\&
 +  C_1
\Bigl[ P^{\frac{1}{d-1}} \bigl| \nabla \rho_0(T^{-1})\bigr| \rho^{-\frac{d}{d-1}}_0 (T^{-1})\Bigr]^d .
\end{align*}
where $C_1$ depends only on $d$.
Let $f$ be of the type
$$
f =-C_2r - \int_{0}^{r} u(s) ds.
$$
One gets
$$
\frac{1}{2} H^{\frac{d}{d-1}}_r
\le
- \frac{C_2}{d-1}
\Bigl[ \frac{P}{ \rho_0( T^{-1})} \Bigr]^{\frac{1}{d-1}}
+
C_1
\Bigl[ P^{\frac{1}{d-1}} \bigl| \nabla \rho_0(T^{-1})\bigr|
\rho^{-\frac{d}{d-1}}_0 (T^{-1})\Bigr]^d.
$$
Then it follows from the assumption of the proposition that the right-hand is negative for a sufficiently large $C_2>0$.
This contradicts the estimate $H_r \ge 0$.

This means that
$$
H_r \exp\bigl(-C_2 r -\int_{0}^{r} u(s) ds \bigr)
$$
can attain its minimum only at $\partial B_R$.
Taking into account that
$$H_r|_{\partial B_R} \ge C_3(C,R) \inf_{x \in \partial A} K(x)$$
one gets the desired estimate.

b) In the proof
we use an idea  from \cite{Tso}.
We are looking for the maximum of
$$
\frac{H_r}{H-g(r)}
$$
on $B_{R} \setminus B_{r_0}$, where $g = \frac{\varepsilon}{2} r$.
Note that $H-g \ge \frac{\varepsilon}{2} r$.
Assume that $\log H_r - \log (H-g(r))$
attains its maximum
at $x_0$ with $|x_0| < R$ (otherwise the estimate is trivial). Then at this point
$$
\frac{H_{rr}}{H_r} - \frac{H_r - g'}{H - g} \le 0, \ \frac{H_{r \theta_i}}{H_r} - \frac{H_{\theta_i} }{H - g} =0.
$$
The second differentiation gives
$$
\frac{H_{r\theta_i\theta_i}}{H_r} \le \frac{H_{\theta_i \theta_i}}{H - g}.
$$
Differentiating the change of variables formula one obtains
$$
\frac{P_r}{P}
=
\frac{H_{rr}}{H_r} + \sum_{i=1}^{d-1}\frac{H_r + H_{r \theta_i \theta_i}}{H + H_{\theta_i \theta_i}}
+ \frac{1}{\rho_0(T^{-1})}
\langle \nabla \rho_0(T^{-1}), T^{-1}_{r} \rangle.
$$
Hence
\begin{align*}
\sum_{i=1}^{d-1} &\frac{H_r + H_{r \theta_i \theta_i}}{H + H_{\theta_i \theta_i}}
 \le
\frac{H_r}{H-g}\sum_{i=1}^{d-1}\frac{ H-g + H_{\theta_i \theta_i}}{H + H_{\theta_i \theta_i}}
=
\frac{H_r}{H-g} \Bigl( d-1  - g \sum_{i=1}^{d-1}\frac{1}{H + H_{\theta_i \theta_i}} \Bigr)
\\&
\le
 (d-1)\frac{H_r}{H-g} \Bigl(  1 - g \sqrt[d-1] {\prod_{i=1}^{d-1}\frac{1}{H + H_{\theta_i \theta_i}}} \Bigr)
\\&
=
(d-1)\frac{H_r}{H-g} \Bigl(  1 - g\sqrt[d-1] H_r \sqrt[d-1] {\frac{\rho_0(T^{-1})}{P}}  \Bigr)
.
\end{align*}

Next using
$$
T^{-1}_r = H_r \cdot {\rm n} +
\sum_{i=1}^{d-1} H_{r \theta_i} \cdot {\rm e_i}
$$
we get
\begin{align*}
&
\frac{1}{\rho_0(T^{-1})}
\langle \nabla \rho_0(T^{-1}), T^{-1}_{r} \rangle
= \frac{\langle \nabla \rho_0(T^{-1}), {\rm n} \rangle}{\rho_0(T^{-1})}  H_r +
 \sum_{i=1}^{d-1} H_{r \theta_i}  \frac{\langle \nabla \rho_0(T^{-1}), {\rm e_i} \rangle }{\rho_0(T^{-1})}
\\&
=
\frac{H_r}{\rho_0(T^{-1})} \Bigl( \langle \nabla \rho_0(T^{-1}), {\rm n} \rangle  +
 \frac{1}{H-g} \sum_{i=1}^{d-1} H_{ \theta_i}  \langle \nabla \rho_0(T^{-1}), {\rm e_i} \rangle  \Bigr).
\end{align*}
Taking into account the assumptions, boundedness of $H_{\theta_i}$ and $H$, we get
$$
\frac{1}{\rho_0(T^{-1})}
\langle \nabla \rho_0(T^{-1}), T^{-1}_{r} \rangle
\le C_1 \frac{H_r}{H-g}.
$$
Thus we obtain
\begin{align*}
& \frac{P_r}{P}
\le
\frac{H_{r}-g'}{H - g} +
(d-1)\frac{H_r}{H-g} \Bigl(  1 - g \sqrt[d-1] H_r \sqrt[d-1] {\frac{\rho_0(T^{-1})}{P}}  \Bigr)
+ C_1 \frac{H_r}{H-g}.
\end{align*}
Multiplying this inequality by $H-g$, using the
assumptions of the theorem and boundedness of $H$
we get
$$
-C(H-g)
\le  C_2 H_r -g'   - (d-1) g H^{\frac{d}{d-1}}_r \sqrt[d-1] {\frac{\rho_0(T^{-1})}{P}}.
$$
Thus implies
\begin{align*}
 H^{\frac{d}{d-1}}_r
 \le
 C_4 \frac{(H-g)}{g}
+ \frac{C_4}{ g} \Bigl( -g' + C_2 H_r \Bigr)
\le  \frac{C_5 + C_6 H_r}{r}
\le \frac{1}{2} H_r^{d/(d-1)} + C_7 \Bigl(\frac{1}{r}\Bigr)^d.
\end{align*}
Hence
$$
\Bigl( \frac{H_r}{H-g} \Bigr)^{d/(d-1)} \le \frac{C_8}{r^{d + \frac{d}{d-1}}}.
$$
This gives the desired result.
\end{proof}

\begin{remark}
The proof of a)
can be generalized to the case
of pre-limiting potentials $H_t$ (see Section 2).
Since the computations are quite involved, we give only some intermediate results.
For simplicity let us  skip the index $t$ and write $H$ instead of $H_t$.
Choose a function $f$ in such a way that
$f(x) \sim -(d-1) \ln(r-r_0)^{+}$ for $x$ close to $\partial B_{r_0}$
and assume that the minimum point $x_0$ does not belong to $\partial B_R $.
One has
$$
T^{-1} = \Bigl(  H + \frac{r}{t+1}  H_r \Bigr) \cdot {\rm n} +
\sum_{i=1}^{d-1} H_{\theta_i} \cdot {\rm e_i}.
$$
The derivatives of $T^{-1}$ at $x_0$ satisfy
$$
T^{-1}_r = \Bigl( \frac{t+2}{t+1} H_r + \frac{r}{t+1}  H_{rr} \Bigr) \cdot {\rm n} +
\sum_{i=1}^{d-1} H_{r \theta_i} \cdot {\rm e_i}.
$$
$$
T^{-1}_{\theta_i} =
\Bigl(  \frac{r}{t+1}  H_{r \theta_i} \Bigr) \cdot {\rm n} +
\sum_{i \ne j} H_{\theta_i \theta_j} \cdot {\rm e_j}.
+ \Bigl(  \frac{t}{t+1} H + \frac{r}{t+1}  H_r + H_{\theta_i \theta_i} \Bigr) {\rm e_i}.
$$
Choosing an appropriate basis, we may assume without loss of generality that
$$
H_{\theta_i \theta_j} = 0
$$
for $i \ne j$.
Then
\begin{align*}
& \det DT^{-1}
=
\Bigl(\frac{t+2}{t+1} H_r + \frac{rH_{rr}}{1+t} \Bigr)
\prod_{i=1}^{d-1} \Bigl[ \Bigl( \frac{t}{1+t} H + \frac{r H_r}{1+t} \Bigr) + H_{\theta_i \theta_i} \Bigr] +
\\&
-
\sum_{i=1}^{d}   \frac{rH^2_{r\theta_i}}{1+t}
\prod_{i \ne j} \Bigl( \Bigl[ \frac{t}{1+t} H + \frac{r H_r}{1+t} \Bigr] + H_{\theta_j \theta_j} \Bigr).
\end{align*}
At the minimum point one has
\begin{equation}
\label{hrt}
\frac{H_{rr}}{H_r} = -f',  \ H_{r \theta_i} =0,
\end{equation}
\begin{equation}
\label{hrrt}
\frac{H_{rrr}}{H_r}  +  (f''-(f')^2) \ge 0, \ H_{r\theta_i \theta_i} \ge 0.
\end{equation}

The reasoning from the above proposition leads to the following estimate:
\begin{align*}
 \frac{( \rho_1)_r}{ \rho_1} + \frac{(d-1) }{r}
& \ge
 H_r \Bigl( \frac{t+2}{t+1} - \frac{r}{t+1} f' \Bigr) \frac{\langle \rm{n}, \nabla \rho_0(T^{-1})\rangle}{\rho_0(T^{-1})}
\\& +
\frac{ -(t+3) f' + r ((f')^2 - f^{''}) }{(t+2) - r f'}
\\& + (d-1) \frac{H^{\frac{d}{d-1}}_{r}}{r} \Bigl( 1 - \frac{rf'}{t+1} \Bigr)
 \Bigl[\Bigl(\frac{t+2}{t+1}  - \frac{r f' }{1+t} \Bigr)\Bigr]^{\frac{1}{d-1}}
 \Bigl[ \frac{\rho_{0}(T^{-1})}{\rho_1} \Bigr]^{\frac{1}{d-1}}
.
\end{align*}
Choosing an appropriate $f$ one gets the desired bound.
\end{remark}

\begin{corollary}
Assume that
$$
P < C, \frac{1}{C} \le \rho_0,
$$
$\partial A$ is smooth and uniformly convex,
$$
 \Bigl|\frac{\nabla \rho_0}{\rho_0}\Bigr|, \Bigl|\frac{P_r}{P}\Bigr| < C
$$
and  $ \rho_0|_{\partial A} \le C$, $\frac{1}{C} \le P|_{\partial B_R}$.
Then $D_1 \varphi^{d} < |\nabla \varphi| < D_2$ for some $D_1,D_2>0$ depending only
on $d, C$ and $\partial A$.
\end{corollary}

\begin{remark}
\label{ap-est-ht} We have proved  the above estimates assuming
smoothness of $H$. But the final results do not depend on the
bounds of the derivatives of $H$. We give some sufficient
conditions for $H$ to be smooth in Section 7. Applying smooth
approximations it is possible to show that the estimates remain
true without extra smoothness assumption of the solution. In
particular, the upper bound on $|\nabla \varphi|$ implies the absence of a
singular part for $D\varphi$.
\end{remark}

\begin{theorem}
\label{sobolev-phi}
Assume that $\rho_{1} =  \frac{C}{r^{d-1}}$. Then
for every $p>0$ there exist $C_{p,R} >0$ such that
\begin{equation}
\label{p-est1}
C_{p,R}
\int_{A} |\nabla \varphi|^{p+1} \ d \mu
\le
 \int_{A}  \Big| \frac{\nabla \rho_{0}}{\rho_{0}}  \Big|^{p+1} \ d \mu
+
\int_{\partial A} |\nabla \varphi|^{p} \rho_{0} \ d \mathcal{H}^{d-1},
\end{equation}
\begin{equation}
\label{p-est2}
C_{p,R}
\int_{A} |\nabla \varphi|^{p+1} \ d \mu
\le
 \int_{A}  \Big| \frac{\nabla \rho_{0}}{\rho_{0}}  \Big|^{p+1} \ d \mu
+
\int_{\partial A} K^{-p} \rho^{p+1}_{0} \ d \mathcal{H}^{d-1}.
\end{equation}
\end{theorem}
\begin{proof}
Under assumptions of the theorem the change of variables formula reads as
$$
C K |\nabla \varphi| = \rho_0.
$$
Computing $DT$ is the standard  frame $\{ {\rm n}, {\rm v_1}, \cdots, {\rm v_{d-1}} \}$, we get
$$
DT =
\left( \begin{array}{cc}
|\nabla \varphi| &  0 \\
b^t  &  A
\end{array}
\right),
$$
where
$$
b = \Bigl( \frac{\varphi \varphi_{{\rm n v_i}}}{|\nabla \varphi|} \Bigr), \ \
A = \Bigl(\frac{\varphi \varphi_{ {\rm v_i v_j}}}{|\nabla \varphi|} \Bigr).
$$
The Jacobian matrix of $S = T^{-1}$
computed $(r,\theta)$ coordinates  has the form
(recall that $\partial_{n} = \partial_r$ and $\partial_{{\rm \theta_i}} = r \partial_{{\rm v_i}}$ )
$$
DS =
\left( \begin{array}{cc}
H_r &  0 \\
c^t  &  B
\end{array}
\right),
$$
with
$$
c = \bigl( H_{r {\rm \theta_i}} \bigr), \ \
B = \Bigl( \frac{H + H_{{\rm \theta_i \theta_j}}}{r} \Bigr).
$$
Recall that $H_r(T) = \frac{1}{|\nabla \varphi|}$.
Since $DS(T) = DT^{-1}$, one can also assume that  $A$ and $B$ are diagonal (at a fixed point).
Denote by $\lambda_i$ the eigenvalues of $A$. Then using
$DT \circ DS(T) = \mbox{Id}$ one easily obtains
$$
\frac{\varphi \varphi_{{\rm n v_i}}}{|\nabla \varphi|^2} + H_{r {\rm \theta_i}}(T) \lambda_i =0.
$$
Next we find
\begin{align*}
\varphi_{nn} & = \partial_n |\nabla \varphi|
= \partial_{n}(1/H_{r}(T)) = - \frac{1}{H^2_r(T)}
\Bigl( H_{rr}(T) |\nabla \varphi|
+ \sum_{i=1}^{d-1} H_{r \theta_i}(T) \langle \partial_{\rm n} {\rm n}, {\rm v_i} \rangle
\Bigr)
\\&
= - \frac{1}{H^2_r(T)}
\Bigl( H_{rr}(T) |\nabla \varphi|
+ \sum_{i=1}^{d-1} H_{r {\rm \theta_i}}(T) \frac{\varphi_{{\rm n v_i}}}{|\nabla \varphi|}
\Bigr)
\\&
= - \frac{1}{H^2_r(T)}
\Bigl( H_{rr}(T) |\nabla \varphi|
- \sum_{i=1}^{d-1} \varphi \frac{\varphi^2_{n v_i}}{\lambda_i |\nabla \varphi|^3}
\Bigr)
\ge
-\frac{H_{rr}(T)}{H^3_r(T)}.
\end{align*}

Taking into account that $\varphi$ has convex level sets (hence $\mbox{\rm div} \frac{\nabla \varphi}{ |\nabla \varphi|} \ge 0$),
we get
$$
\mbox{\rm div} \Bigl( \varphi \frac{\nabla \varphi}{ |\nabla \varphi|} |\nabla \varphi|^p \Bigr)
\ge
|\nabla \varphi|^{p+1} + p \varphi |\nabla \varphi|^{p-1} \varphi_{nn}
\ge
|\nabla \varphi|^{p+1}
-p\frac{r H_{rr}}{H^{p+2}_r} \circ T.
$$
Thus
\begin{equation}
\label{div-est}
|\nabla \varphi|^{p+1}
\le
\mbox{\rm div} \Bigl( \varphi \frac{\nabla \varphi}{ |\nabla \varphi|} |\nabla \varphi|^p \Bigr)
+ p
\frac{r H_{rr}}{H^{p+2}_r} \circ T .
\end{equation}
Integrate (\ref{div-est}) over $A$ with respect to
$\mu$. One obtains
\begin{align*}
\int_{A} & \mbox{\rm div} \Bigl( \varphi \frac{\nabla \varphi}{ |\nabla \varphi|} |\nabla \varphi|^p \Bigr) \rho_0 \ dx
=
R \int_{\partial A} |\nabla \varphi|^{p} \rho_0 \ d \mathcal{H}^{d-1}
-
\int_{A} \varphi \frac{\langle \nabla \varphi, \nabla \rho_0 \rangle }{ |\nabla \varphi|} |\nabla \varphi|^p  \ dx
\\&
\le
R \int_{\partial A} |\nabla \varphi|^{p} \rho_0 \ d \mathcal{H}^{d-1}
+ \varepsilon \int_A |\nabla \varphi|^{p+1} d \mu
+ N(\varepsilon,p)
\int_{A} \varphi^{p+1} \Bigl|\frac{\nabla \rho_0}{\rho_0}\Bigr|^{p+1}  \ d \mu.
\end{align*}

Applying the change of variables and integrating by parts  we get
\begin{align*}
p \int_{A} &
\frac{r H_{rr}}{H^{p+2}_r} \circ T \ d \mu
=
p \int_{B_R}
\frac{r H_{rr}}{H^{p+2}_r} \ d \nu
=
-\frac{p}{p+1} \int_{B_R} \langle \nabla H^{-p-1}_r, x \rangle \rho_1 \  dx
\\& =
-\frac{pR}{p+1} \int_{\partial B_R}  H^{-p-1}_r \rho_1 d \mathcal{H}^{d-1}
+
\frac{p}{(p+1)} \int_{B_R}  H^{-p-1}_r \rho_1 \ \ dx
\\&
=
\frac{p}{(p+1)} \int_{A}  |\nabla \varphi|^{p+1}  \  d \mu
-\frac{pR}{p+1} \int_{\partial A}  |\nabla \varphi|^{p} \rho_0 \ d  \mathcal{H}^{d-1}.
\end{align*}
The obtained estimates imply immediately
(\ref{p-est1}). Estimate (\ref{p-est2}) follows from (\ref{p-est1})
and the change of variables formula.
\end{proof}

\begin{remark}
Estimates of these type are also available for the pre-limiting potentials.
For instance, for $T = \varphi \frac{\nabla \varphi}{|\nabla \varphi|} |\nabla \varphi|^{\frac{1}{1+t}}$
one has
$$
\varphi_{\rm{nn}}
 \ge
 \Bigl( \frac{1}{p} |\nabla \varphi(S)|^p\Bigr)_r \circ T
 $$
 for $p = 2 + \frac{1}{1+t}$.
Then one can show that for $q>0$
$$
 \int_{\partial A}  \varphi |\nabla \varphi|^{q} \rho_{0} \ d \mathcal{H}^1 +
\int_A \Bigl| \frac{\nabla \rho_{0}}{\rho_0} \Bigr|^{q+1} \varphi^{q+1} \ d \mu
\ge
C_q \int_{A}   |\nabla \varphi|^{q+1} d \mu.
$$ 
\end{remark}

\begin{remark}
The result can be easily generalized  to the general case of a
continuous rotational invariant density
$\nu = \rho_{\nu} \ dx = \rho_{\nu}(r) \ dx$.

Indeed, take a mapping $T$ sending $\nu$ to $\frac{dx}{r}$
and having the form $T(x) = f(r) \frac{x}{r}$. The function $f$ satisfies
$$
r \rho_{\nu}(r) = f'(r).
$$
Note that $\psi \frac{\nabla \psi}{|\nabla \psi|}$, where $\psi = f(\varphi)$ sends
$\mu$ to $\frac{dx}{r}$. Applying (\ref{p-est1}) to $\psi$
we get
$$ C
\int_{A} \varphi^{p+1} \rho^{p+1}_{\nu}(\varphi) |\nabla \varphi|^{p+1} \ d \mu
\le
  \int_{A}  \Big| \frac{\nabla \rho_{0}}{\rho_{0}}  \Big|^{p+1} \ d \mu
 +
\int_{\partial A} |\nabla \varphi|^{p+1} \rho_{0} \ d \mathcal{H}^{d-1}.
$$
\end{remark}

\begin{remark}
It looks possible to prove $L^{\infty}$-bounds on $|\nabla \varphi |$  using the parabolic maximum principle (see the next Section)  and assuming high integrability of $|\nabla \rho_0|$.
 Estimates of this type for the potential $u$ have been obtained in
\cite{Guit2}.
Results from \cite{Guit2} are not directly applicable
to our situation, since we need to consider $u$ in unbounded domains.
\end{remark}

\section{Variants of the parabolic maximum principle}

For every convex $V$
we denote by $|\partial V|(B)$ the associated Monge--Amp{\`e}re  measure of the set $B$,
which is defined as follows:
$$
|\partial V|(B) = \lambda \bigl( \{ \bigcup_{x \in B} \partial V(x) \} \bigr),
$$
where $\partial V$ is the subdifferential of $V$ at $x$.

For  smooth  $V$ one has
$$
\partial V = \det D^2 V \ dx.
$$
This means that $\nabla V$
sends $\partial V$ to Lebesgue measure if $\det D^2 V \ne 0$.

Recall that for every continuous function $f$ on a convex set $A$ one can define its convex envelope
$f_{*}$  which is the supremum of all affine functions  less than $f$.
The set $\mathcal{C}_{f} = \{x \colon  f(x) = f_{*}(x)\}$ is called the set of contact points of $f$.

According to the  elliptic maximum principle (also called Alexandrov maximum principle or Alexandrov-Bakelman-Pucci principle)  every
continuous function $f$ on a convex set $A \subset \R^d$
satisfies
$$
\sup_{A} f \le \sup_{\partial A} f + C \cdot \mbox{diam}(A) \Bigl[
\partial f_*(\mathcal{C}_{f}) \Bigr]^{\frac{1}{d}},
$$
where $C$ depends only on $d$.
If $f$ is
twice continuously differentiable, this implies
$$
\sup_{A} f \le \sup_{\partial A} f + C \cdot \mbox{diam}(A) \Bigl[ \int_{D^2 f(x) \le 0}
| \det D^2 f(x) | dx\Bigr]^{\frac{1}{d}},
$$
where $C$ depends only on $d$.
Equivalently, passing to $g = \sup_{A} f - f$,
one gets that for every non-negative $g$
$$
\inf_{\partial A} g \le  C \cdot \mbox{diam}(A) \Bigl[ \int_{D^2 g(x) \ge 0}  \det D^2 g(x)  dx\Bigr]^{\frac{1}{d}}.
$$

A parabolic version of the maximum principle
was obtained by Krylov (see \cite{Kryl}).
Later Tso \cite{Tso2}
simplified the proof in some special cases and gave  extensions in some particular cases.

In this section we prove some other variants of the parabolic maximum principle.

\begin{definition}
For a continuous function $f$ defined on a convex set $A$ consider
its sublevel set $A_t = \{f \le t\}$ and the convex envelope $\mbox{conv}(A_t)$ of $A_t$.
Every point $x \in \mbox{Int} A$ satisfying $x \in A_t \cap \partial \mbox{conv}(A_t)$
for some $t$ we call a contact point of $A_t$. The set of
all such points will be denoted by $\mathcal{C}_{f, l}$.
\end{definition}

We denote by $S^{d-1}_+$ the upper half of the unit sphere in $\R^d$.
For every set
$
\Omega = \Bigl\{ (r, \theta) \colon  \  R_1 \le r \le R_2, \ \theta \in Q \Bigr\},
$
where $Q \subset S^{d-1}_+$ is a spherically convex set, we denote by
$$
\partial_p \Omega = Q \times R_2 \cup \partial Q \times [R_1,R_2]
$$
its parabolic boundary.

\begin{theorem}
\begin{itemize}
\item[1)]
Let $v$ be a twice continuously differentiable function on a convex set $A \subset \R^d$.
Then there exists a constant $C=C(d)$ depending only on $d$ such that
\begin{equation}
\label{parmp}
 \sup_{x \in A} v(x)
\le \sup_{x \in \partial A} v(x)
+ C(d)
\int_{\mathcal{C}_{-v, l}} |\nabla v| K dx.
\end{equation}
where $K(x)$ is the Gauss curvature of the set $\partial {\mbox{ {\rm conv}}} \{y \colon  v(x) \le v(y) \}$ at $x$
\item[2)]
Let $\Omega$ be a set of the type
$$
\Omega = \Bigl\{ (r, \theta) \colon  \  R_1 \le r \le R_2, \ \theta \in Q \Bigr\}
$$
with a spherically convex $Q \subset S^{d-1}_{+}$ satisfying $\mbox{{\rm dist}} (Q, \partial S^{d-1}_{+}) > 0$.
Then for every twice continuously differentiable function $f \colon \Omega \to \R$ satisfying $\sup_{x \in \partial_p \Omega} f  \ge 0$, one has
\begin{equation}
\label{parmpsp}
\sup_{\Omega} f \le C_1 \sup_{\partial_p \Omega} f
+ C_2 \Bigl[ \int_{ \Gamma_f} \frac{|f_r \det (f \cdot \mbox{\rm Id} + D^2_{\theta} f)|}{r^{d-1}}  dx \Bigr]^{\frac{1}{d}},
\end{equation}
where $\Gamma_f = \{x \in \Omega \colon  f_r \le 0, \ f \cdot \mbox{\rm Id} + D^2_{\theta} f \le 0 \}$,
and constants $C_1, C_2 >0$ depend only on $d$ and  $Q$.
\end{itemize}
\end{theorem}
\begin{proof}

1) Set $f=(M-v)^{1/d}$, where $M = \sup_{A} v$. The estimate (\ref{parmp}) is equivalent to
\begin{equation}
{\inf_{\partial A} f^{d}}
\le  C \int_{\mathcal{C}_{f, l}}
f^{d-1} |\nabla f| K dx.
\end{equation}
For every $0< t < \inf_{\partial A} f$
let us consider the set $A_t = \{x \colon  f(x) \le t\} \subset A$ and its convex envelope
${\mbox {\rm conv}} (A_t)$.
Since $A$ is convex,  ${\mbox {\rm conv}} (A_t)$ lies inside of $A$ and, in addition,  $\mbox{dist}({\mbox {\rm conv}} (A_t),\partial A)>0$.
Set: $C_t = \partial {\mbox {\rm conv}} (A_t) \cap A_t$. Since $f$ is smooth, the image of $C_t$ under the Gauss map  $\mbox{\rm{n}}$ of $\partial {\mbox {\rm conv}} (A_t)$
covers the unit sphere. Hence the image of
$$\bigcup_{0 < t < \inf_{\partial A} f} C_t = \mathcal{C}_{f, l}$$  under $T = f \cdot \mbox{\rm n}$
coincides with $\{x \colon  \| x \| \le \inf_{\partial A} f \}$. One has $\det DT = f^{d-1} |\nabla f| K$.
The result follows  from the change of variables formula.

2)
Let us consider the set of vectors
$V$ satisfying
$$
a) \ \big\langle v, \mbox{\rm n} \big\rangle < M, \ \mbox{for all} \ x \in \Omega
$$
and
$$
b) \ \big\langle v, \mbox{ \rm n} \big\rangle > m, \ \mbox{for all} \ x \in \partial_p \Omega,
$$
with $\mbox{ \rm n}  =  \frac{x}{|x|}$, $M = \sup_{x \in A} f$, $m = \sup_{x \in \partial A} f$.
Since $\mbox{{\rm dist}} (Q, \partial S^{d-1}_{+}) > 0$, the set of vectors $v$ satisfying
$b)$ is non-empty and has the form
$$
(r, \theta) \colon   \ \ r > C(Q) m, \ \theta \in \tilde{Q}
$$
for some set $\tilde{Q} \subset S^{d-1}_+$ and a constant $C(Q)$ depending on $Q$.
If $M < C(Q) m$, the claim is proved. If not, then $V$ is nonempty.
Consider the set
$$
B = \{ (r, \theta) \colon   \ \  C(Q) m < r <M, \ \theta \in \tilde{Q}\} \subset V.
$$
Clearly,
$$
C_0(Q) (M - C(Q)m)^d \le \lambda(B).
$$
It remains to estimate $\lambda(B)$.
For every $a \in B$ define $M_a = \{x \colon  f(x)=\langle a, n \rangle\}$.
Conditions a) and b) imply that $M_a$ is non-empty and
 contained inside of $\Omega$. Hence, there exists a point $x_0 \in M_a$
in the interiour of $\Omega$, where $|x|$ attains its maximum.
One has at this point
$$
f(x_0) = \langle a, \mbox{\rm n} \rangle,
$$
$$
f_{v}(x_0) = \langle a, \mbox{\rm n}_v \rangle =  \frac{1}{|x_0|} \langle a, v \rangle.
$$
for every unit $v \bot {\rm n}$.
This implies that
$$
D_{\theta} f = a  - \langle a, \mbox{\rm n} \rangle \mbox{\rm n}.
$$
In addition,
$$
f_{r}(x_0) \le 0, \ \ D^2_{\theta} f(x_0) \le  D^2_{\theta} \big\langle a, \mbox{\rm n} \rangle
= - \langle a, \mbox{\rm n} \rangle \cdot \mbox{\rm Id}  = -f(x_0) \cdot \mbox{\rm Id} .
$$
Hence $B \subset \Gamma_f$. Set:
$$
S = f(x) {\mbox \rm{n}} + |x| \sum_{i=1}^{d-1} f_{{\rm v_i}}(x) {\rm v_i}
=  f(x) {\mbox \rm{n}} + D_\theta f(x).
$$
Note that $S(x_0) =a $. This means that
$$
S(\Gamma_f) = B.
$$
By the change of variables formula
$$
\lambda(B) \le
\int_{\Gamma_f} \det DS \ dx
=\int_{\Gamma_f} \frac{|f_r \det (f \cdot \mbox{\rm Id} + D^2_{\theta} f)|}{r^{d-1}}  dx.
$$
The proof is complete.
\end{proof}

\begin{remark}
\label{sphere-plane}
Inequality (\ref{parmpsp}) implies a form of
the parabolic maximum principle (see \cite{Tso2}).
Assume that $\sup_{\partial_p \Omega'} f =0$. Then
\begin{equation}
\label{parmpsp+}
\sup_{\Omega'} f \le C  \Bigl[ \int_{ \Gamma_f} \frac{|f_r \det (f \cdot \mbox{\rm Id} + D^2_{\theta} f)|}{r^{d-1}}  dx \Bigr]^{\frac{1}{d}}.
\end{equation}
Set $u=\sqrt{1+z^2} f$, where $z$ and $x$ are related by the change of variables described
is Section 2. Using
$$\det \Bigl(f \cdot \mbox{Id} + D^2_{\theta} f \bigr) = \det \Bigl((1+z^2)^{3/2} D^2_z u \Bigr)$$
and trivial uniform estimates one gets
\begin{equation}
\label{Tsopm}
\sup_{\Omega} u \le   C(d,Q) \Bigl[ \int_{\Gamma_u \cap \Omega} |u_t \cdot \det D^2_x u|  \ dt  dx \Bigr]^{\frac{1}{d}},
\end{equation}
$\Gamma_u = \{u_t \le 0; D^2 u \le 0\}$,
for any $u$ with $\sup_{\partial_p \Omega} u =0$ and a cylinder $\Omega = [0,R] \times Q$
with convex $Q$.
To remove the restriction $\sup_{\partial_p \Omega} u =0$
one applies  the estimate to $u=v- \sup_{\partial_p \Omega} v$.
\end{remark}

\begin{remark}
The above variants of the parabolic maximum principle
are naturally related with transport mappings of the type
$$
T=\varphi \frac{\nabla \varphi}{|\nabla \varphi|}, \ \
S = H \cdot \mbox{\rm n} + D_{\theta} H.
$$
Both variants of mappings  can be obtained from the
"elliptic"  transportation  $\nabla V$ by scaling procedures (see Section 2).
The transportation by gradients are naturallly associated with the elliptic maximum principle.
Is it possible to derive both parabolic maximum principles from the elliptic one?

1) {\bf Elliptic maximum principle implies (\ref{parmp})}.

We prove that for every continuous $f \ge 0$  on a convex set $A \subset \R^d$
satisfying $\inf_{x \in A} f(x)=0$ and
every $0  < p  \le 1$ there exists a constant $C=C(d)$ depending only on $d$ such that
$$
{\inf_{\partial A} f^{d(1+p)}}
\le C \mbox{ \rm diam}^{dp}(A) \ |\partial W_{*}| \circ S_{p}^{-1} \Bigl( \mathcal{C}_{W}\Bigr),
$$
where $\mathcal{C}_{W}$ is the set of contact points of $W=\frac{p}{p+1} f^{1+ \frac{1}{p}}$
and $S_p(x) = \frac{x}{|x|^{1-p}}$.
In particular, if $f$ is twice continuously differentiable, one has
\begin{equation}
\label{p-ABP}
{\inf_{\partial A} f^{d(1+p)}}
\le C \mbox{ \rm diam}^{dp}(A) \int_{\{x \colon  D^2 f^{1+ \frac{1}{p} }(x)\ge 0 \}} \det
D \Bigl(  f \frac{\nabla f}{|\nabla f|^{1-p}} \Bigr) dx.
\end{equation}
Clearly, letting $p \to 0$ we deduce an equivalent form of (\ref{parmp})  from (\ref{p-ABP}).

{\bf Proof:} Let $x_0$ be a point satisfying $f(x_0)=0$.
If $x_0 \in \partial A$ there is nothing to prove. Thus we assume that $x_0 \notin \partial A$.
Let $V$ be the convex function whose graph is the upside-down cone with vertex $(x_0,0)$
and base $A$ with $V=m$ on $\partial M$, where $m =
\inf_{x \in \Omega} \frac{p}{p+1} f^{1+ \frac{1}{p}}(x)$.
It is easy to check that
$$
B_{m/ \mbox{\rm diam} (\Omega)}
\subset
\partial V(x_0) \subset \partial W_{*}(\mathcal{C}_{W}).
$$
Note that the measure with density
$$
\rho  \colon = \det D \Bigl( \frac{x}{|x|^{1-p}}\Bigr)
$$
is the image of  Lebesgue measure under $S_p$.
Hence
$$
\int_{B_{m/ \mbox{\rm diam} (\Omega)}} \rho dx
\le
|\partial W_{*}| \circ S_{p}^{-1} \Bigl( \mathcal{C}_{W}\Bigr).
$$
The direct computation yields
$
\rho = p \ r^{d(p-1)}.
$
This immediately gives
$$
\tilde{C}  \Bigl( \frac{m}{\mbox{\rm diam} (\Omega)}\Bigr)^{dp}
=
\int_{B_{m/ \mbox{\rm diam} (\Omega)}} \rho \ dx
\le
|\partial W_{*}| \circ S_{p}^{-1} \Bigl( \mathcal{C}_{W}\Bigr)
$$
with $\tilde{C}$ depending only on $d$.
This proves the first part.

Finally, (\ref{p-ABP}) can be  obtained by direct computations. We just notice that
$\{x \colon  D^2 f^{1+ \frac{1}{p} }(x)\ge 0 \} \subset \mathcal{C}_{W}$.

2) {\bf Elliptic maximum principle implies (\ref{parmpsp}) with $\Omega=B_R$ and symmetric $f$}.

Let $f  \colon  B_R \to \R$ be a symmetric ($f(-x)=f(x)$)
bounded function. Assume that $f$ is twice continuously differentiable
at every $x \ne 0$ and  $\inf_{x \in B_r} f(x) \le 0$.
Then there exists a constant $C=C(d)$ depending only on $d$ such that
\begin{equation}
\label{dualABP}
\inf_{\partial B_R} f
\le C \Bigl[\int_{ \Gamma_{-f}} \frac{f_r \det (f \cdot \mbox{\rm Id} + D^2_{\theta} f)}{r^{d-1}}  dx\Bigr]^{\frac{1}{d}}.
\end{equation}

{\bf Proof:} For every $t>0$ consider  $$w_t(x) = |x| \ f(x |x|^{-\frac{t}{1+t}})$$
defined on $B_{R^{1+t}}$.
One has
$$
R^{1+t} \inf_{\partial B_R} f = \inf_{z \in \partial B_R} w_t(|z|^{1+t})
= \inf_{\partial B_{R^{1+t}}} w_t.
$$
Since $w_t(0)=0$, by the elliptic maximum principle
$$
\inf_{\partial B_R} f
=
 \Bigl( \frac{ \inf_{\partial B_{R^{1+t}}} w_t }{R^{1+t}} \Bigr)
 \le C(d) \Bigl( \int_{\mathcal{C}_{w_t}} \det D^2 w_t \ dx \Bigr)^{\frac{1}{d}}.
$$
Indeed,   $w_t$ is twice continuously differentiable everywhere in $B_r$ except, maybe, the point $x=0$.
Without loss of generality one can assume that $\inf_{B_R} w_t <0.$
Since $w_t$ is continious, $\inf w_t$ is attained at some point $\tilde{x}$. Since $w_t$ is symmetric, the points
$(\tilde{x}, w_t(\tilde{x}))$
and $(-\tilde{x}, w_t(\tilde{x}))$ belong to a horisontal supporting hyperplane
to the graph of $w_t$. Since $w_t(0)=0$,  clearly $0 \notin {\mathcal{C}_{w_t}}$.
This justifies the above estimate.

Set: $S_t(y) =y|y|^t$. Then
\begin{align*}
\int_{\mathcal{C}_{w_t}} \det D^2 w_t \ dx
&
=
\int_{\mathcal{C}_{w_t}(S_t) \subset B_R} \det D^2 w_t (S_t) \det D S_t \ dy
\\& =
\int_{\mathcal{C}_{w_t}(S_t) \subset B_R} \det D \bigl( \nabla w_t (S_t) \bigr) dy .
\end{align*}
Direct computations
yield
$$
\nabla w_t (S_t) =  \Bigl( f + \frac{r f_r}{1+t}  \Bigr) \cdot {\rm{n}} + D_{\theta} f.
$$
Hence
$$
 \lim_{t \to 0} \det D \bigl( \nabla w_t (S_t) \bigr)
 = \frac{f_r \det (f + D^2_{\theta} f)}{r^{d-1}}.
$$
The proof is complete.

3) {\bf Does elliptic maximum principle imply (\ref{Tsopm})?}
There are good reasons to believe that the elliptic maximum
principle implies (\ref{Tsopm}).
This problem seems to be rather involved technically and we do not
consider it here. We just give a proof  in a  particular simple  case.
Let $f$ satisfy all the assumptions from item~2). In addition, assume that
$f=C$ outside of $\Omega' \cup (-\Omega')$, where
$$
\Omega' = \Bigl\{ (r, \theta) \colon  \  0 \le r \le R, \ \theta \in Q' \Bigr\}
$$ and
$Q' \subset S^{d-1}_{+}$ satisfies $\mbox{{\rm dist}} (Q', \partial S^{d-1}_{+}) > 0$.
By the previous result
$$\inf_{\partial \Omega'} f
\le C \Bigl[\int_{ \Gamma_{-f} \cap \Omega'} \frac{f_r \det (f \cdot \mbox{\rm Id} + D^2_{\theta} f)}{r^{d-1}}  dx\Bigr]
^{\frac{1}{d}}.
$$
Arguing as in Remark \ref{sphere-plane} we get that
$$
\inf_{\partial_p \Omega} u \le   C(d,Q) \Bigl[ \int_{\Gamma_{-u}} |u_t \cdot \det D^2_x u|  \ dt  dx \Bigr]^{\frac{1}{d}},
$$
holds
for any bounded $u \colon  (0,R] \times \R^{d-1}$, satisfying
\begin{itemize}
\item[1)] $u$ is smooth on $(\varepsilon,R] \times \R^{d-1}$
\item[2)] $u$ is constant outside of
$\Omega = (0,R] \times Q$ with convex  $Q \subset \R^{d-1}$
\item[3)] $\inf_{\Omega} u = 0$.
\end{itemize}
Passing to $u = \sup_{\Omega}\varphi - \varphi$ we obtain
$$
\sup_{\Omega} \varphi \le  C(d,Q) \Bigl[ \int_{\Gamma_{\varphi}} |\varphi_t \cdot \det D^2_x \varphi|  \ dt  dx \Bigr]^{\frac{1}{d}}
$$
for any smooth compactly supported $\varphi$ with $\mbox{supp} (\varphi) \subseteq \Omega$.
\end{remark}

\section{Isoperimetric inequality}

We discuss two apparently
different proofs of the isoperimetric Euclidean inequality for convex sets (it is well-known that the general case
can be easily reduced to the convex one).
First of them due Gromov. It is worth mentioning (this was pointed out to
the author by S.~Bobkov)
that  arguments of such type go back to Knothe \cite{Kno}.
More precisely, it has been shown in \cite{Kno} that the Brunn--Minkowsky inequality can be proved
by  transportation arguments with the help of triangular mappings.
The second proof comes from the differential geometry.
Our aim is to reveal a remarkable similarity between probabilistic and geometrical
points of view.

1) ({\bf Mass transportation. Probabilistic approach.})
We  follow the mass transportation arguments but use the
Gauss mass transport instead of optimal (or triangular) one.
Let $A \subset \R^d$ be a convex set
and $T = \varphi \frac{\nabla \varphi}{|\nabla \varphi|}$
send Lebesgue measure on $A$ into  Lebesgue measure on
$B_{R}$, where $B_R$ is a ball of the same volume.
By the change of variables (see the previous section)
$$
\varphi^{d-1} |D_a \varphi| K=1.
$$
Hence by the arithmetic--geometric inequality
$$
1 = \det D_a T \le \frac{1}{d-1} \ \mbox{Tr} D_a \Bigl(  \varphi \frac{\nabla \varphi}{|\nabla \varphi|}\Bigr),
$$
where $D_a T$ is the absolutely continuous part of the distributional derivative
$D T$.
Clearly,
$$
\lambda(A) \le \frac{1}{d-1} \int_{A} \mbox{\rm div} \ T \ dx
= \frac{1}{d-1} \int_{\partial A} \langle T, {\mbox {\rm n}} \rangle \ d\mathcal{H}^{d-1}
\le \frac{R}{d-1} \mathcal{H}^{d-1}(\partial A).
$$
Taking into account that $\lambda(A) = \lambda(B_R) = c_d R^{d}$, one easily recovers
 the classical isoperimetric inequality.

2) ({\bf Curvature flows. Geometric approach.})
The same proof can be rewritten in the language of curvature flows. The curvature flow proofs are
well-known in differential geometry (see  partial results on the
Cartan--Hadamard conjecture in \cite{Top}, \cite{Schulze}). Let $A_t = \{x \colon  \varphi(x) \le t\}$.
For convenience we assume that $\varphi$ is smooth on $\{x \colon  \varphi(x) >0\}$ (which is indeed the case for
smooth strictly convex
$\partial A$).
Note that $A_t$ are expanding with the speed $\frac{1}{|\nabla \varphi|}$.
The enclosed volume $\lambda (A_t)$ evolves with the  speed which
can be exactly computed by the Gauss--Bonnet theorem
$$
\frac{d}{dt} \lambda(A_t) =  t^{d-1}  \int_{\partial A_t} K \ d\mathcal{H}^{d-1} = t^{d-1} \mathcal{H}^{d-1}(S^{d-1}).
$$
Hence
$$
\lambda(A_t) = \lambda(B_t).
$$
In the other hand, it is known that
$$
\frac{d}{dt} \mathcal{H}^{d-1}(\partial A_t) =  t^{d-1}  \int_{\partial A_t} K H \ d\mathcal{H}^{d-1},
$$
where $H$ is the mean curvature. By the arithmetic-geometric inequality
$K^{1/(d-1)} \le \frac{H}{d-1} $. Hence
$$
\frac{d}{dt} \mathcal{H}^{d-1}(\partial A_t) \ge (d-1) t^{d-1}
 \int_{\partial A_t} K^{\frac{d}{d-1}} \ d\mathcal{H}^{d-1}.
$$
By  H{\"o}lder's inequality
\begin{align*}
&\frac{d}{dt} \mathcal{H}^{d-1}(\partial A_t) \\& \ge (d-1)t^{d-1}
\Bigl( \int_{\partial A_t} K\ d\mathcal{H}^{d-1} \Bigr)^{\frac{d}{d-1}} \Bigl(\mathcal{H}^{d-1}(A_t)\Bigr)^{-1/(d-1)}
\\&
= (d-1)t^{d-1}
\Bigl( \mathcal{H}^{d-1}(S^{d-1}) \Bigr)^{\frac{d}{d-1}} \Bigl(\mathcal{H}^{d-1}(A_t)\Bigr)^{-1/(d-1)}.
\end{align*}

Integrating in $t$ one obtains
$$
\mathcal{H}^{d-1}(A_t) \ge t^{d} \mathcal{H}^{d-1}(S^{d-1}) = \mathcal{H}^{d-1}(B_t).
$$
The proof is complete.

\section{On H{\"o}lder's regularity of the Gauss mass transport}

The elliptic and parabolic Monge--Amp{\`e}re  equations
belong to the family of the so-called fully nonlinear
PDE's.
See \cite{CafCab} (and \cite{Guit}
for the special case of the Monge--Amp{\`e}re  equation).
A  short survey \cite{Kryl2}
presents the developments of the main ideas
of the nonlinear PDE's theory.

The connection between
the variational Monge--Kantorovich problem and the elliptic
Monge--Amp{\`e}re  equation was revealed by Brenier
(see \cite{Vill}). In \cite{Tso} the existence of the Gauss curvature flow
for smooth data
was established by solving the corresponding equation of the parabolic Monge--Amp{\`e}re  type.

Contributions to the regularity theory of the elliptic Monge--Amp{\`e}re  equation
were made by many authors, including Alexandrov, Calabi, Yau, Pogorelov, Krylov, Spruck, Caffarelli,
Nirenberg, and Urbas.
There are several approaches
to the regularity theory of nonlinear equations.
A classical one is based on
 differentiating
of the underlying equation.  Taking the second derivative
one obtains another equation
which is linear with respect to higher derivatives.
Then one applies a~priori estimates from the
linear theory. This was a common way
for studying the nonlinear PDE's before the
results of Krylov, Safonov, and Evans on a~priori
estimates for nonlinear uniformly elliptic operators.
See \cite{Kryl2} for  details.

Unfortunately, the elliptic Monge--Amp{\`e}re operator
$$
u \to \det D^2 u
$$
is {\bf not}
uniformly elliptic even in the class of convex functions.
This is the reason why the Krylov--Safonov--Evans theory is not applicable directly.
 The regularity problem for the elliptic Monge--Amp{\`e}re  equation
 was solved in sufficient generality by Caffarelli.
 Combining the nonlinear regularity theory and deep geometric considerations he proved, in particular,
 that the solution $V$ of the optimal transportation problem
 $$
 g(\nabla V) \det D^2 V = f
 $$
 for probability measures $f \  dx$ and $g \ dy$
 with compact supports $X$ and $Y$ is $(2+\alpha)$-H{\"o}lder continuous inside of $X$
 provided $f,g$ are H{\"o}lder continuous, bounded away from zero and
 $Y$ is convex.

Many regularity results for the Gauss curvature flows
 (see \cite{Tso}, \cite{BenAnd})
 were obtained  by using the classical way
 of differentiating the evolution equation.
 Similar to the the elliptic case,  the parabolic maximum principle
 plays a crucial role in the study of this problem.

A parabolic analog of regularity theory for uniformly  operators has been developed in
\cite{LW}.

The regularity of the parabolic Monge--Amp{\`e}re  equation was studied by Krylov \cite{Kryl},
 Ivochkina, Ladyzhenskaya \cite{IvLad}, Guti{\'e}rrez, Huang \cite{Guit2},
  R.H.~Wang and G.L.~Wang \cite{WW1}, \cite{WW2} (see \cite{Kryl2}
  for references).
  Some interesting results were proved by probabilistic methods (optimal control and stochastic differential equations), see \cite{Kryl2}, \cite{Spil}.
 A parabolic analog of the
 Caffarelli theory for the  elliptic Monge--Amp{\`e}re
was developed by R.H.~Wang and G.L.~Wang in \cite{WW1}, \cite{WW2}.
 They studied the  parabolic Monge--Amp{\`e}re  equation
 \begin{equation}
 \label{pmabp}
 u_t \det D^2_z u = f(t,z)
 \end{equation}
 on the domain
 $Q= \Omega \times [0,T]$
 with given values $u=\varphi(t,z)$ on the parabolic boundary $\partial_p Q$.
 It was shown in \cite{WW1} that under the assumptions that

 \begin{itemize}
 \item[1)]
 $\Omega$ is compact, strictly convex with $C^2$-boundary
 \item[2)]
 $f$ is positive Lipschitz continuous on $\overline{Q}$
 \item[3)]
 $\varphi \in C^{2,1}(\overline{Q})$ with $\varphi_t >0$, $D^2_z \varphi >0$
 on $\overline{Q}$
 \end{itemize}
 there exists a  solution $u \in C^{1+\alpha/2, 2+ \alpha}_{loc}$ for some $\alpha>0$.
A measure-theoretic interpretation (the
parabolic Monge--Amp{\`e}re  measure) was given in \cite{WW2}.

For further generalizations and refinements, see  \cite{LiChen}.
Sobolev estimates for (\ref{pmabp}) are obtained in \cite{Guit2}.

We analyze below the regularity result of~Tso.
The reasoning from \cite{Tso} can be easily generalized to our situation
with the help of our results from Section 4. We will not
repeat the lengthy reasoning from \cite{Tso}
and give just a brief sketch of the proof.

Let ${\tilde C}^{k,\alpha}(B_R)$ be the parabolic H{\"o}lder norm (see \cite{Tso})
on functions $$f(r,\theta)  \colon  [0,R] \times S^{d-1} \mapsto \R.$$
It was established in  \cite{Tso} that for $\rho_0=1$, $\rho_1 =\rho_1(\theta) \in C^{2+\alpha}(S^{d-1})$ with some $\alpha >0$
and every $R > r_0 >0$ there exists $C$ such that
$$
|H_r|_{\tilde{C}^{\beta/2, \beta}(B_{R}\setminus B_{r_0})}
+
|D^2_\theta H|_{\tilde{C}^{\beta/2, \beta}(B_{R}\setminus B_{r_0})} <C
$$
for some $\beta>0$.

Using the estimate $0< c_{r_0}<H_r<C_{r_0}$ for smooth $H$ from Section 4  and repeating the arguments from \cite{Tso}
it is not hard to verify Theorem \ref{Tso1} below, which is a generalization of Theorem
 4.2 from \cite{Tso}. Clearly, a solution $H$
obtained in this theorem coincides with the potential $H$ of the corresponding Gauss mass transport
by the uniqueness theorem from \cite{BoKo}.

\begin{theorem}
\label{Tso1}
Assume that $\rho_1 \in  C^{2,\alpha}(B_R), \rho_0 \in C^{2,\alpha}(A)$, $A$ is uniformly convex and
$H_A \in C^{2,\alpha}(S^{d-1})$.
Then a solution $H$ to (\ref{MA0+}) with $H|_{\partial A}=H_A$ exists.
In addition, $$H \in \tilde{C}^{4,\varepsilon}(B_R \setminus B_{r_0})$$ for every $r_0$ and
\begin{equation}
\label{hoelder}
|H_r|_{\tilde{C}^{\beta/2, \beta}(B_{R}\setminus B_{r_0})}
+
|D^2_\theta H|_{\tilde{C}^{\beta/2, \beta}(B_{R}\setminus B_{r_0})} <C
\end{equation}
holds for  some positive $\beta, \varepsilon, C$ depending on $r_0$, $R$, the
curvature of $\partial A$,
the H{\"o}lder and uniform bounds on $\rho_0$, $\rho_1$.
\end{theorem}
{\bf Sketch of the proof:}
One proves the existence of a  solution to (\ref{MA0+}).

1) The classical short-time existence result implies that a
smooth ($\tilde{C}^{4,\varepsilon}$) solution to (\ref{MA0+}) with a given initial value $H(R,
\theta)=H_{A}(\theta)$  exists for $t \in [R-\varepsilon, R]$ (see,
for instance \cite{Gerh}, Theorems 2.5.7, 2.5.9). Let $[R^*,R]$ be
the maximal existence interval. Assume that $R^{*}>0$. Applying the change of variables formula, let us 
estimate the volume enclosed by
the hypersurface determined by $H(R^*, \theta)$.  One concludes
that there exists a sphere contained in all hypersurfaces
determined by $H(r,\theta)$, $r>R^*$. Taking the center of this
sphere as the new origin one can assume without loss of generality
that $H$ is strictly positive on $[R^*,R]$.

2) The results of Section 4 give  $\frac{1}{C} < |H_r| <C$ for
some $C>0$ and every $r \in [R^*, R]$. Following \cite{Tso} one obtains
that $\Lambda > H + D^2_{\theta} H > \lambda$ for some constants
$0 < \lambda < \Lambda$, $t \in [R^*, R]$. This can be shown by
differentiating twice the equation in $\theta$ and applying the classical maximum
principle to a suitable function (see also Pogorelov-type
arguments in \cite{GilTrud}, Theorem 17.19).
Alternatively, one can use
Caffarelli's result  \cite{Caff} on bounds for principal
curvatures of a smooth convex set under the assumption that  the corresponding
Gauss curvature is positive and bounded.

3) Differentiate (\ref{MA0+}) in $r$.
The Krylov--Safonov estimates (see \cite{Kryl}) imply that
the parabolic H{\"o}lder norm of $H_r$ on $[R^*,R]$ is under control. The same holds for  $H_{\theta}$.

4) It remains to prove H{\"o}lder's continuity of $H + D^2_{\theta} H$.
The arguments follow \cite{Tso}.
Let us indicate the main difference.
To estimate oscillation of $H_{\theta \theta}$ (or $u_{zz}$) we  need an estimate for
the additional term
$$
\big|\log \rho_0(T^{-1})(x) - \log \rho_0(T^{-1})(y) \big|
\le \sup \frac{|\nabla \rho_0|}{\rho_0} |T^{-1}(x) - T^{-1}(y)|.
$$
(See \cite{Tso}, Theorem 4.1, (4.3)-(4.4)).
Then we estimate $T^{-1}(x) - T^{-1}(y)$ by a parabolic H{\"o}lder norm of $H_{\theta}$ (see item 3)).
Thus we get
$$
H_{\tilde{C}^{1+\delta/2, 2+\delta}} \le C \Bigl[ \bigl( H_{\theta}  \bigr)_{{\tilde C}^{\varepsilon/2, \varepsilon}} + 1 \Bigr]
$$
for some $\delta, \varepsilon >0$.
Then the parabolic interpolation inequalities (see \cite{KrylHoeld}, Theorem 8.8.1.) complete the proof.

5)
Since we have managed to keep control on the norms of derivatives of $H$
on $[R^*,R]$, the solution  exists for $r<R^*$ by the short-time existence theorem.
Hence $R^*=0$. The proof is complete.

This work was supported by the RFBR projects 07-01-00536 and 08-01-90431-Ukr,
RF President Grant MD-764.2008.1,
and the SFB701 at the University of Bielefeld.

\end{document}